\documentclass[11pt]{article}

\usepackage[english]{babel}
\usepackage[utf8x]{inputenc}
\usepackage{amsmath}
\usepackage{dsfont}
\usepackage{graphicx}
\usepackage{amsthm}
\usepackage{amsfonts}
\usepackage{fullpage}
\usepackage[toc,page]{appendix}

\newtheorem{Lemma}{Lemma}
\newtheorem{Corollary}{Corollary}
\newtheorem{Theorem}{Theorem}
\newtheorem{Definition}{Definition}
\numberwithin{Definition}{section}
\numberwithin{Theorem}{section}
\numberwithin{Lemma}{section}
\numberwithin{Corollary}{section}

\theoremstyle{remark}
\newtheorem*{Remark}{Remark}
\newtheorem*{Assumption}{Assumption}

\newcommand{\dom}[0]{\text{dom}}
\newcommand{\E}[0]{\mathcal{E}}
\newcommand{\R}[0]{\mathds{R}}

\title{Extensions and their Minimizations on the Sierpinski Gasket}
\author{Pak Hin Li\thanks{Research supported by the Mathematics Department of the Chinese University of Hong Kong and the Undergraduate Summer Research Fellowship from the office of Academic Links of the Chinese University of Hong Kong. }, Nicholas Ryder\thanks{Research supported by the National Science Foundation through the Research Experiences 
for Undergraduates(REU) Program at Cornell.},  \\ Robert S. Strichartz\thanks{Research supported by the US NSF grant DMS-1162045.}, Baris Evren Ugurcan\thanks{Research supported by the US NSF grant DMS-1156350.}}

\begin{document}\maketitle

\abstract{We study the extension problem on the Sierpinski Gasket ($SG$).  In the first part we consider minimizing the functional $\mathcal{E}_{\lambda}(f) = \mathcal{E}(f,f) + \lambda \int f^2 d \mu$ with prescribed values at a finite set of points where $\mathcal{E}$ denotes the energy (the analog of $\int |\nabla f|^2$ in Euclidean space) and $\mu$ denotes the standard self-similiar measure on $SG$. We explicitly construct the minimizer $f(x) = \sum_{i} c_i G_{\lambda}(x_i, x)$ for some constants $c_i$, where $G_{\lambda}$ is the resolvent for $\lambda \geq 0$.  We minimize the energy over sets in $SG$ by calculating the explicit quadratic form $\mathcal{E}(f)$ of the minimizer $f$. We consider properties of this quadratic form for arbitrary sets and then analyze some specific sets. One such set we consider is the bottom row of a graph approximation of $SG$. We describe both the quadratic form and a discretized form in terms of Haar functions which corresponds to the continuous result established in a previous paper. In the second part, we study a similar problem this time minimizing $\int_{SG} |\Delta f(x)|^2 d \mu (x)$ for general measures.  In both cases, by using standard methods we show the existence and uniqueness to the minimization problem. We then study properties of the unique minimizers. 

\noindent \emph{Keywords:} Sierpinski Gasket, Whitney extension theorem, extension from finite set of data, conductance, piecewise harmonic function, piecewise biharmonic function, Haar functions.

\noindent \emph{2010 MSC subject classification:} 28A80.
}

\section{Introduction}

In classical analysis one often wants to study finite variants of the Whitney extension theorem, in which data at a finite set of points in Euclidean space (or domains in Euclidean space or Riemannian manifolds) is prescribed, such as values and certain derivatives of a function. The problem is to find a function with the prescribed data that minimizes (or comes close to minimizing) some prescribed norm. There are many such problems depending on the nature of the data and the chosen norm. We give just a few references to recent work in this area (\cite{F}, \cite{F2}, \cite{FI}).\\

The purpose of this paper is to initiate the study of analogous problems on fractals. For the most part we restrict attention to the "poster child" of all fractals, the Sierpinski gasket (SG). Some of our results are "generic" and extend easily to Kigami's class of post-critically finite (PCF) fractals \cite{Ki}\cite{Str}. But we also include a number of results that deal specifically with the geometry of $SG$. We are interested in minimizing two Sobolev types of norms. The first, treated in section 2, is $\E(u) + \lambda \int |u|^2 d \mu$ for some fixed $\lambda \geq 0$, where $\E$ denotes the standard self-similar energy on $SG$. For $\lambda = 0$ this amounts to minimizing energy, while for $\lambda>0$ we are minimizing the analog of the $H^1$ Sobolev norm in Euclidean space. For this minimization we prescribe values of the function on a fixed finite set $E$. We prove existence and uniqueness of a minimizer, and show how to express the solution in terms of the resolvent of the Laplacian, which was studied in detail in \cite{resolvent}. In the case of $\lambda =0$ the minimizers are analogs of continuous piecewise linear functions on the line, and in fact are continuous piecewise harmonic functions on the complement of $E$. The energy of the minimizer is expressible as $$\E(u) = \sum_{\{x,y\}\subset E}c_{x,y}(u(x)- u(y))^2$$ for certain "conductance" coefficients $c_{x,y}$ depending on the set $E$. In the case when the points in $E$ are all junction points in $SG$, we give estimates for the size of the coefficients, and characterize the pairs of points with $c_{x,y} = 0$ in terms of path connectedness properties of $E$. We study in detail three specific families of sets $E$ for which we compute the coefficients explicitly. The most challenging of those we call the "bottom row". For fixed $m$, we take $E$ to be the equally spaced $2^m + 1$ points along the bottom line in $SG$. The results we obtain show that this is a good discrete approximation of the continuous problem of computing the energy of functions that are harmonic in the complement of the bottom line that was studied in \cite{OS}.\\

The second type of Sobolev norm, considered in section 3, is 
$$T_{\zeta}(u) = \int |\Delta_\zeta u|^2 d \zeta$$
where $\zeta$ is any finite positive continuous measure on $SG$ and $\Delta_\zeta$ is the associated Laplacian (\cite{Ki},\cite{Str}). We are most interested in the cases $\zeta = \mu,$ the standard self-similar measure, and $\zeta = \nu$, the Kusuoka measure. This norm is the analog of the Sobolev $H^2$ norm on the line. Here we consider two types of data: prescribing function values on  a finite set $E$, or prescribing those values and also the values of normal derivatives at a set $F$ of junction points $(F=E$ is most interesting). For the first problem we prove existence and uniqueness of the minimizer under the additional geometric hypothesis that the set $E$ is not contained in a straight line in the harmonic coordinates realization of $SG$. A simple counterexample shows that some such hypothesis is needed for uniqueness. We obtain a similar existence and uniqueness theorem for the second type of data. For the analogous problem on the line, the minimizers are $C^1$ piecewise cubic splines, so it is not surprising that on $SG$ the minimizers are piecewise biharmonic functions (solutions of $\Delta_\zeta^2u = 0$). In the case where $\zeta = \mu$ and $E$ consists of all junction points of a fixed level, we are able to give an explicit local formula for $T_\mu(u)$ in terms of the data. If we consider the minimization of $T_\mu(u)$ with just the values prescribed on all junction points of a fixed level, the problem is no longer local. In this case we are able to give an explicit linear formula relating the discrete Laplacian and the continuous Laplacian for minimizers. In principle all that is needed to obtain an explicit solution is to solve these equations, but we do not have any insight into the form of the solution. In a similar vein, we are able to express the minimizer for any data set $E$ in terms of integrals of the Green's function, but again this involves solving a certain system of linear equations.\\ \\
We mention a couple of interesting open problems relating to this work:

1) Minimize the functional $\|\Delta u\|_\infty$ with prescribed values and normal derivatives. Of course the minimizer will not be unique. The analogous problem on the line is local and was solved in \cite{G} in terms of piecewise quadratic functions with one additional cut point in each interval between points in $E$. On $SG$ the analog of quadratic functions is solutions of $\Delta u = constant$. It is not at all clear what "piecewise" should mean in this context.

2) Study analogous problems on non-PCF fractals, such as the Sierpinski carpet or products of $SG$.

\section{Minimizing Generalizations of Energy}
In this section we start with a straightforward example which sets up many methods used later. We are concerned with the space $\dom \E$.

We consider the functional $\E_\lambda: u \mapsto \mathcal{E}(u) + \lambda \int u^2 d \zeta$ defined on $\dom \E$ for $\lambda \geq 0$. Notice that given $u \in \dom \mathcal{E}$ we have that $u$ is continuous and thus is bounded in this space so $\E_\lambda(u) < \infty$.

We apply the following well known theorem:

\begin{Theorem} \label{hilbertconvexminimum}
If $X$ is a non-empty closed convex subset of a Hilbert space $H$, there exists a unique point $y \in X$ that minimizes the norm of all points in $X$.
\end{Theorem}

We first establish the following fact which we use throughout the paper.
\begin{Lemma}\label{lemma:supenergy}
Given $u \in \dom\E$, if $u(q) = 0$ for some $q \in SG$, then 
\begin{equation}\label{equation:supenergy}
\|u\|^2_\infty \leq C \mathcal{E}(u).
\end{equation}
\end{Lemma}
\begin{proof}
We first note the following	fact (check \cite{Str} for reference)
$$|u(x) - u(y)| \leq \mathcal{E}(u)^{1/2}R(x,y)^{1/2}$$
where $R(x,y)$ denotes the resistance metric. Since $SG$ is compact with respect to effective resistance metric, this implies
$$\|u\|^2_\infty \leq C \mathcal{E}(u).$$
\end{proof}
Now we set $X = \{ u \in \dom \E ~|~u(x_i) = a_i\}$ for all $i$ where $\{x_1, \dots, x_n\}$ is any set in $SG$ . We consider the Hilbert space $\dom \E / \text{const.}$ and the projection $\tilde{X}$ of X to this space. 
\begin{Lemma}
For $\lambda = 0$, $\tilde{X}$ is closed and convex in $ \dom \E / \text{const.}$.
\end{Lemma}
\begin{proof}
The convexity follows immediately. We consider a sequence $u_n \to u$ with $u_n \in \tilde{X}$. We pick representatives $\widetilde{u_n}$ of $u_n$ such that $\widetilde{u_n} \in X$ and $\tilde{u}$ of $u$ such that $\tilde{u}(x_1) = a_1$. We can apply Lemma \ref{lemma:supenergy} to $\tilde{u} - \widetilde{u_n}$. Since $\E(\tilde{u} - \widetilde{u_n}) \to 0$ we get $\|\tilde{u}-\widetilde{u_n}\|_\infty \to 0$. Thus $\tilde{u} \in X$ so $\widetilde{X}$ is closed.
\end{proof}
\begin{Lemma}
For $\lambda > 0$, $X$ is closed and convex in the Hilbert space $(\dom \E, \E_\lambda)$.
\end{Lemma}
\begin{proof}
The convexity follows immediately. \\

We consider a sequence $u_n \to u$ with $u_n \in X$. We set $v_n = u_n - u_n(x_1) = u_n - a_1$ and $v = u - u(x_1)$. Since $\E_\lambda(u_n - u) \to 0$ we get $\E(u_n - u) \to 0$. Applying Lemma \ref{lemma:supenergy} to $v_n - v$ we get $\|v_n - v\|_\infty \to 0$. \\

Since $\E_\lambda (u_n-u) \to 0$ we also have $\int (u_n - u)^2 d \zeta \to 0.$
\begin{align*}
\int(u_n - u)^2 d \zeta &= \int ((v_n - v) + (a_1 - u(x_1)))^2 d \zeta \\
&= \int (v_n - v)^2 d \zeta + 2(a_1 - u(x_1))\int(v_n - v) d \zeta + \int (a_1 - u(x_1))^2 d \zeta .\\
\end{align*} 
Since  $\|v_n - v\|_\infty \to 0$, the above implies $a_1 = u(x_1)$. Thus we can use Lemma \ref{lemma:supenergy} to $u_n - u$ to see $\|u_n - u\|_\infty \to 0$, so $X$ is closed.
\end{proof}

Thus we get both existence and uniqueness of the minimum. 

\subsection{Construction}
To find an explicit construction, we use the functions $G_{\lambda} $ constructed in \cite{resolvent}. Thus we restrict our attention to the case where we have Dirichlet boundary conditions and we only consider the standard measure. We use the notation $\dom \E_0 = \{u \in \dom \E ~|~ u|_{V_0} \equiv 0\}$.\\

First we establish a necessary condition that parallels the Euler Lagrange equations in smooth analysis. 
\begin{Lemma} \label{lemma:eulerlagrange}
Suppose that we have $u$ which minimizes $\E_\lambda$. Then we get\begin{equation}\label{equation:eulerlagrange}\E(u,v) = -\lambda \int u v d\mu\end{equation} for all $v \in \dom \E_0$ with $v|_E \equiv 0.$
\end{Lemma}
\begin{proof}
Suppose $u$ minimizes $\E_\lambda$ with respect to the constraints.
Let $v \in \dom{\E}_0$ with $v|_E \equiv 0.$ Given $t \in \R$ we have $u+tv \in \text{dom}\E_0$ with $u+tv|_E \equiv u|_E$. We compute
$$\E_\lambda(u+tv) = \mathcal{E}(u+tv,u+tv) + \lambda \int (u+tv)^2 d \mu$$
$$ = \mathcal{E}(u) + 2 t \E(u,v) + t^2 \E(v) + \lambda\left(\int u^2 d\mu + 2 t \int u v d \mu + t^2 \int v^2 d \mu\right).$$
Since $u$ is a minimizer, if we view the function $f(t)=\E_\lambda(u+tv)$, we must have $t = 0$ at a minimum. Thus we apply single variable calculus to show the stated result.
\end{proof}
Now we use the functions $G_{\lambda}$ constructed in $\cite{resolvent}.$ A quick calculation shows that $G_{\lambda}$ satisfies (\ref{equation:eulerlagrange}).
\begin{Lemma} \label{lemma:construction}
(Construction) Given a set $E = \{ x_1, x_2, \dots, x_n \} \subset SG$ and $a_i$ then we can guarantee unique $c_i$ such that $f(x) = \sum_i c_i G_\lambda(x,x_i)$ satisfies $f(x_i) = a_i.$ Furthermore this function is the unique minimizer of $\E_\lambda.$
\end{Lemma}
\begin{proof}
We have	a necessary	condition for the minimizer. We	show that there	exists a unique	function satisfying	this neccessary	conditions.	Then the existence of	a minimizer	guarantees this function	is in fact the minimizer.
First we show uniqueness. Suppose $u_1, u_2$ both satisfy (\ref{equation:eulerlagrange}), then we get $$\E(u_1 - u_2) = -\lambda \int (u_1 - u_2)^2 d\mu.$$ Thus we have $\E(u_1 - u_2) = 0$ which means they must differ by a harmonic function, but the Dirichlet boundary conditions guarantee that this harmonic function is identically zero so $u_1 = u_2$.   \\

To show existence, we set up the following linear system:
Let $G$ be defined by $[G]_{i,j} = G_\lambda(x_j,x_i).$ Now setting a vector of $c_i$ and $a_i$ we obtain the system
$$G c = a.$$
This is a linear map from equal dimension vector spaces. We look at the kernel of this map. If we have $G y = 0$ then we know $\tilde{f}(x) = \sum y_i G_\lambda(x,x_i)$ with $\tilde{f}(x_i) = 0$.
Note that this function	satisfies (\ref{equation:eulerlagrange}) for all appropriate $v$ so	we can apply uniqueness to show $\tilde{f} \equiv 0$	since the constant zero	function also satisfies	(\ref{equation:eulerlagrange}).
By Fundamental Theorem of Linear Algebra, injectivity implies surjectivity so we can guarantee the desired $c_i$. \\

Thus we have found the unique function satisfying (\ref{equation:eulerlagrange}) which therefore is the minimizer.
\end{proof}

Now we compute the value of this minimizer:
$$ \E_\lambda(u) = \sum_{i,j} c_i c_j \E_\lambda(G_\lambda(x,x_i),G_\lambda(x,x_j))$$
$$ = \sum_{i,j} c_i c_j G_\lambda(x_i,x_j) = \sum_i c_i a_i$$
$$ = a \cdot c = G^{-1}a \cdot a.$$

\begin{Remark} In fact, we can construct the minimizer when we do not consider the Dirichlet boundary conditions. We sketch the construction here: \\

For $\lambda \ge 0$, $\lambda$ is not a Dirichlet eigenvalue so the space of eigenfunctions is always three dimensional and it is completely determined by the boundary values on $V_0$. We denote the eigenfunctions by $u_i$ with $u_i(q_i)=1, u_i(q_{i-1})=0, u_i(q_{i+1})=0$ and $\Delta u_i=\lambda u_i$.\\

In the case $E\subset \text{SG}\setminus V_0$, let $u=\sum_it_iu_i(x)+\sum_j c_j G_{\lambda}(x_j,x)$. For all $v\in \text{dom} \E$ with $v|_E=0$, we have $\E_{\lambda}(u,v)=\sum_{k}v(q_k)\left(\sum_i t_i \partial_n u_i(q_k)+\sum_j c_j \partial_n G_{\lambda}(q_k,x_j)\right)$. Now suppose that there exist $t_i,c_j$ such that $u=\sum_it_iu_i(x)+\sum_j c_j G_{\lambda}(x_j,x)$, $u(x_i)=0$ and $\partial_nu(q_k)=0$. Then we have $\E_{\lambda}(u,v)=0$ for all $v$ with $v|_E=0$. This would imply that $u=0$.\\

Given $a_i$, we can always construct a function $u$ such that $u(x)=\sum_it_iu_i(x)+\sum_j c_j G_{\lambda}(x_j,x)$, $u(x_i)=a_i$ and also $\sum_i t_i \partial_n u_i(q_k)+\sum_j c_j \partial_n G_{\lambda}(q_k,x_j)=0$ for $k=0,1,2$. Then from the uniqueness property we can guarantee that this $u$ is the unique minimizer.

\end{Remark}

\subsection{Arbitrary Sets of Junction Points}

Now we restrict our attention to arbitrary sets $E \subset V_m$ and the standard energy i.e. $\lambda=0$. We are guaranteed that given the minimizer $u$ we have
$$\E(u) = \sum_{\{x,y\}\subset E}{c_{x,y}(u(x)-u(y))^2}.$$
We seek some basic properties for the coefficients $c_{x,y}$, which may be interpreted as conductances regarding $E$ as an electrical resistance network.

\subsubsection{Notation}
Throughout this section we consider different minimizing sets $\Gamma' \subset \Gamma \subset V_m$. \\
Given a minimizing set $\Gamma$, we let $c^\Gamma_{x,y}$ denote the coefficient for $(u(x)-u(y))^2$ in the minimizing form.  If there is no ambiguity about which set we are looking at, we will use the shorthand $c^\Gamma_{x,y} = c_{x,y}$. 
\begin{Definition}
We define an electrical path in $\Gamma' \subset V_m$ to be a sequence of points $x_1 \to ... \to x_n$ where $x_i \in \Gamma'$ and $c^{\Gamma'}_{x_i,x_{i+1}} > 0$
\end{Definition}

\subsubsection{Estimates on Restricted Coefficients}
\begin{Lemma} \label{lemma:subgraph}
Given any $\Gamma' \subset \Gamma \subset V_m$ and points $x, y \in \Gamma' \subset \Gamma$, then $c^{\Gamma'}_{x,y} \geq c^{\Gamma}_{x,y}$.
\end{Lemma}

\begin{proof}
It suffices to show this for the case where $\Gamma'$ and $\Gamma$ differ by a point. Let $\Gamma = \{x_1, ..., x_{n+1}\}$ and $\Gamma' = \{x_1, ..., x_n\}$ and let $u(x_i) = y_i$ so we have 
$$E(u) = \sum_{1 \leq i < j \leq n}{c_{i,j}(y_i - y_j)^2} + \sum_{i = 1}^{n}{c_{n+1,i}(y_i - y_{n+1})^2}.$$
We differentiate with respect to $y_{n+1}$ to find the conductances for the subgraph.
$$\partial_{y_{n+1}} E(u) = \sum_{i=1}^{n}{2 c_{i,n+1} (y_{n+1} - y_i)} = 0~~\text{so}$$
$$y_{n+1} \sum_{i=1}^{n}{c_{i,n+1}}= \sum_{i=1}^{n}{c_{i,n+1} y_i}~~\text{so}$$
$$y_{n+1} = \frac{\sum_{i=1}^{n}{c_{i,n+1} y_i}}{ \sum_{i=1}^{n}{c_{i,n+1}}}.$$
Thus the energy for the subgraph is

$$E(u) = \sum_{1 \leq i < j \leq n}{c_{i,j}(y_i - y_j)^2} + \sum_{i = 1}^{n}{c_{n+1,i}(y_i - \frac{\sum_{j=1}^{n}{c_{j,n+1} y_j}}{ \sum_{j=1}^{n}{c_{j,n+1}}})^2}.$$
To simplify notation let the points be $x_1$ and $x_2$. We want to show the conductance increases.
We have 
$$c^{\Gamma'}_{x_1, x_2} = \frac{-1}{2}\partial_{y_1y_2}E(u)$$
$$= c_{1,2} + c_{1,n+1}\left(1-\frac{c_{1,n+1}}{\sum_{i=1}^{n}{c_{i,n+1}}}\right)\left(\frac{c_{2,n+1}}{\sum_{i=1}^{n}{c_{i,n+1}}}\right) + c_{2,n+1}\left(1-\frac{c_{2,n+1}}{\sum_{i=1}^{n}{c_{i,n+1}}}\right)\left(\frac{c_{1,n+1}}{\sum_{i=1}^{n}{c_{i,n+1}}}\right)$$
$$ - \sum_{j=3}^{n}\left(c_{j,n+1}\frac{c_{1,n+1}}{\sum_{i=1}^{n}{c_{i,n+1}}}\frac{c_{2,n+1}}{\sum_{i=1}^{n}{c_{i,n+1}}}\right)$$
$$ = c_{1,2} + \frac{c_{1,n+1}c_{2,n+1}}{(\sum_{i=1}^{n}{c_{i,n+1}})^2}\left(\sum_{i=1}^{n}{c_{i,n+1}} - c_{1,n+1} + \sum_{i=1}^{n}{c_{i,n+1}} - c_{2, n+1} - \sum_{i=3}^{n}{c_{i,n+1}}\right)$$
$$ = c_{1,2} + \frac{c_{1,n+1}c_{2,n+1}}{\sum_{i=1}^{n}{c_{i,n+1}}}.$$
Thus we obtain
\begin{equation}\label{equation:conductance}
c^{\Gamma'}_{x_1, x_2}  = c_{1,2} + \frac{c_{1,n+1}c_{2,n+1}}{\sum_{i=1}^{n}{c_{i,n+1}}}.
\end{equation}

\end{proof}

\begin{Corollary} \label{corollary: neighbors}
Given any subset $A\subseteq V_m$, if $ x,y \in A$ such that $x \underset{m}{\sim} y$, then 
$$(\frac{5}{3})^m \leq c^A_{x,y} \leq R(x,y)^{-1} \leq 4(\frac{5}{3})^m$$where $R(x,y)$ denotes the effective resistance between $x$ and $y.$
\end{Corollary}

\begin{proof}
We have that $\Gamma_{x,y} \subset A \subset V_m$ where $\Gamma_{x,y} = \{x, y\}$ is the restriction to just $x$ and $y$. On $V_m$, by Lemma \ref{lemma:subgraph}, we have 
$$(\frac{5}{3})^m = c^{V_m}_{x,y} \leq c^A_{x,y}.$$ 

We also have  $$c^A_{x,y}\leq c^{\Gamma_{x,y}}_{x,y} = R(x,y)^{-1} $$ where $R(x,y)$ is the effective resistance metric so we have $R(x,y)^{-1} \leq 4(\frac{5}{3})^m$  from (1.6.6) in \cite{Str} since $x \underset{m}{\sim} y$.

\end{proof}

\subsubsection{Zero Coefficients}
Throughout this section we let $\Gamma = V_m$ be our ambient space.
We set out to prove the following theorem:
\begin{Theorem} \label{zerocoefficients}

Given $x, y \in \Gamma' \subset \Gamma,$ $c^{\Gamma'}_{x,y} = 0$ if and only if every path in $\Gamma$ from $x$ to $y$ intersects $\Gamma'$.
\end{Theorem}

\begin{Lemma}
If $x, y \in \Gamma' \subset \Gamma$ and every path from $x$ to $y$ intersects $\Gamma'$, then $c^{\Gamma'}_{x,y} = 0$.
\end{Lemma}

\begin{proof}
Let $\Gamma' = \{x,y, x_1, x_2, x_3, \dots, x_n \}$. We define the propogation set of y as
\[
P(y): = \{g \in \Gamma~|~ \text{there is a path P from } y \text{ to } g \text{ with } P \cap \Gamma' = \{y\} ~\}~.
\]
Now we partition the edges as follows:
$$F = \{\{a,b\} ~|~ a \sim b \text{ with one of } a \text{ or } b \text{ in } P(y) \}$$
$$G = \{\{a,b\} ~|~ a \sim b \text{ with } a,b \notin P(y) \}~.$$
From the definition of $P(y)$ we immediately get the fact that given $e_F \in F, e_G \in G$ if $x \in e_F \cap e_G$ then $x \in \Gamma'$. Thus we have for each $z \in \Gamma \setminus \Gamma'$ the edges containing $z$ are contained either all in $F$ or all in $G$. Thus we get a partition of $\Gamma \setminus\Gamma'$ into
$$\Gamma'_F = \{x \in \Gamma \setminus\Gamma' | x \in z \text{ for some } z \in F \}$$
$$\Gamma'_G = \{x \in \Gamma \setminus\Gamma' | x \in z \text{ for some } z \in G \}.$$
Then we can express the energy as
\[
\E(u) =\left( \frac{5}{3} \right)^m \left( \sum_{F} (u(a) - u(b))^2 + \sum_{G} (u(m) - u(n))^2 \right)~.
\]

Therefore, since we have partition of $\Gamma \setminus \Gamma'$ we can minimize each sum individually to get functions $\E_F$ and $\E_G$ on $\Gamma'$ with
\[
\E(u) = \E_F + \E_G
\]
where $\E_F$ depends on $u(y)$, not on $u(x)$ and $\E_G$ depends on $u(x)$ not on $u(y)$.
Hence, $\partial_{u(x)u(y)} \E(u) = 0$. So, $c_{x,y}^{\Gamma'} = 0$.
\end{proof}

To prove the converse of this, we will use a reduction argument to deal with the situation where $\Gamma'$ differs from $\Gamma$ by one point. We establish the appropriate machinery to make this reduction:
\begin{Lemma} \label{lemma:gamma}
Given $\Gamma = \{x_1, x_2, \dots, x_n, x_{n+1} \}$ and $\Gamma' = \{ x_1, x_2, \dots, x_n\}$, suppose that $c_{1, n+1}^{\Gamma}> 0$ and $c_{n+1, 2}^{\Gamma}>0$.  Then, $c_{1, 2}^{\Gamma'} >0$.
\end{Lemma}

\begin{proof}
This follows from (\ref{equation:conductance}), since
\[
c_{1, 2}^{\Gamma'} = c_{1, 2}^{\Gamma} + \frac{c_{1,n+1}^{\Gamma} c_{2,n+1}^{\Gamma}}{\sum_{j=1}^{n} c_{j,n+1}^{\Gamma}} >0~.
\]
\end{proof}
\begin{Lemma} \label{lemma: electrical path}
Consider an electrical path $x_{i_0} \rightarrow x_{i_1} \rightarrow x_{i_2} \rightarrow x_{i_3} \rightarrow \dots \rightarrow x_{i_n}$ and $p \in \Gamma \setminus \{x_{i_0}, x_{i_n}\}$. Denote $\Gamma'$ the restricted graph for $\Gamma \setminus \{p\}$. If $p$ is not on the path, then we have an electrical path of the same length and with the same points.  If $p=x_{i_j}$ is on the path, then we have shorter electrical path in $\Gamma'$, $x_{i_0} \rightarrow x_{i_1} \rightarrow x_{i_2}  \rightarrow \dots \rightarrow x_{i_{j-1}} \rightarrow x_{i_{j+1}} \rightarrow \dots \rightarrow x_{i_n}$. 
\end{Lemma}

\begin{proof}
If $p$ is not in the path then we know from Lemma \ref{lemma:subgraph} that $c^{\Gamma'}_{i_j,i_{j+1}} \geq c^{\Gamma}_{i_j,i_{j+1}} > 0$ so  the same points form an electrical path. \\

If $p = x_{i_j}$ is on the path then it follows from the above observation that we only need to show $c^{\Gamma'}_{i_{j-1},i_{j+1}} > 0$. Since $c^{\Gamma}_{i_{j-1},i_{j}}, c^{\Gamma}_{i_j,i_{j+1}} > 0$ by Lemma \ref{lemma:gamma} we have $c^{\Gamma'}_{i_{j-1},i_{j+1}} > 0$. We therefore have a shorter path $x_{i_0} \rightarrow x_{i_1} \rightarrow x_{i_2}  \rightarrow \dots \rightarrow x_{i_{j-1}} \rightarrow x_{i_{j+1}} \rightarrow \dots \rightarrow x_{i_n}$.
\end{proof}

With this tool, we can easily make the reduction and use our previous calculation to show the converse:
\begin{Lemma}
Let $\Gamma' = \{x_1, x_2, \dots, x_n \}$ and $\Gamma  = \{x_1, x_2, \dots, x_n, x_{n+1}, x_{n+2}, \dots, x_{n+k} \}$.
Let $x,y$ be two distinct points in $\Gamma'$. Suppose there is a path from $x$ to $y$ such that the path does not intersect any  other points in $\Gamma'$. Then $c_{x,y}^{\Gamma'}>0$.
\end{Lemma}

\begin{proof}
We prove here a stronger result, namely that this holds for all electrical paths satisfying our criteria. By Corollary \ref{corollary: neighbors}, we know every path is an electrical path, so it suffices to show this. \\

For notation we set $x=x_{i_0}$ and $y=x_{i_t}$. Let $x_{i_0} \rightarrow x_{i_1} \rightarrow x_{i_2} \rightarrow \dots \rightarrow x_{i_t}$ be a path connecting $x_{i_0}$ and $x_{i_t}$ where all $x_{i_j}$ are in $\Gamma \setminus \Gamma'$,where $1\le j \le t-1$. Define $\Gamma^{(j)}=\{x_1, x_2, \dots, x_n, x_{n+1},\dots,x_{n+j} \}$. Then  from Lemma \ref{lemma: electrical path}, we can inductively get an electrical path from $x$ to $y$ which has length less than or equal to $j+1$ for the restricted graph $\Gamma^{(j)}.$ \\

We have $c_{x,y}^{\Gamma'} \ge c_{x,y}^{\Gamma^{(1)}}$. As a result, we only need to concentrate on the case $\Gamma^{(1)}.$ This case follows directly from Lemma \ref{lemma:gamma}. 
\end{proof}
\subsubsection{Lower bounds}

\begin{Lemma}
Let $x_0\rightarrow x_1 \rightarrow x_2 \rightarrow \dots \rightarrow x_n \rightarrow x_{n+1} $ be a path in $V_m$ such that $x_i \neq x_j$ for all $i \neq j$ and $x_i \underset{m}{\sim} x_j$ if and only if $|i-j|=1$. Define 
$V_m^0 = V_m$ and
$V_m^i = V_m \setminus \{x_1, \dots ,x_i\}~.$
We use the notation $c^i_{z_1,z_2} = c^{V_m^i}_{z_1,z_2}~$, and let
$$N_i= \# \{y\in V_m^{i-1}~|~c_{y,x_i}^{i-1}>0\}~,$$

$$M_i=\max \{c_{y,x_i}^{{i-1}}~|~ y\in V_m^{i-1}, c_{y,x_i}^{i-1} >0\},$$

$$m_i=\min \{c_{y,x_i}^{{i-1}}~|~ y\in V_m^{i-1}, c_{y,x_i}^{{i-1}} >0\}~.$$
Then we obtain the relationships
$$1 \le N_i \le i+3,$$ $$\left(\frac{5}{3}\right)^m\frac{m_i}{N_i M_i}\le m_{i+1},$$  $$4\left(\frac{5}{3}\right)^m\left(1+\frac{M_i}{N_i m_i}\right) \ge M_{i+1}.$$

\end{Lemma}

\begin{proof}
Let $y$ be a point such that $c_{y,x_i}^{{i-1}} > 0$.\\

We claim that either $c_{y,x_{i+1}}^{{i-1}} = 0$ or $y \underset{m}{\sim} x_{i+1}$. To see this suppose $c_{y,x_{i+1}}^{{i-1}} > 0$ and they are not neighbors. Then we know we have a path from $y$ to $x_{i+1}$ where the interior of the path is contained in $\{x_1, \dots, x_{i-1}\}$ by Theorem \ref{zerocoefficients}. Now we know the second point in the path and $y$ are neighbors, implying that we have $x_j \underset{m}{\sim} x_{i+1}$ for some $j\leq i-1$, contradicting our choice of path. Thus either they are neighbors or $c_{y,x_{i+1}}^{{i-1}} = 0$.\\

From (\ref{equation:conductance}) we have:

\begin{align*}
c_{y,x_{i+1}}^{{i}}=c_{y,x_{i+1}}^{{i-1}}+ \frac{c_{y,x_{i}}^{{i-1}}c_{x_{i+1},x_{i}}^{{i-1}}}{{\sum_{z\in V_m^i}c_{x_{i},z}^{{i-1}}}}~.\end{align*}

It follows that
\begin{align*}
4\left( \frac{5}{3} \right)^m + \frac{4\left( \frac{5}{3} \right)^m M_i}{N_i m_i} \geq c_{y,x_{i+1}}^{i} \geq \frac{\left( \frac{5}{3} \right)^m m_i}{N_i M_i}.
\end{align*}

So we have $$\left(\frac{5}{3}\right)^m\frac{m_i}{N_i M_i}\le m_{i+1}$$ and $$4 \left(\frac{5}{3}\right)^m\left(1+\frac{M_i}{N_i m_i}\right) \ge M_{i+1}~.$$

To show $1 \le N_i \le i+3$, note that we obviously have $1 \le N_i$ by Corollary \ref{corollary: neighbors}. By Lemma \ref{lemma:gamma} and Corollary \ref{corollary: neighbors} we obtain $N_{i+1} \le N_i +1$. By induction, it follows that $N_i \leq i+3$. Thus we get the stated bounds.

\end{proof}

\begin{Theorem}
There is a sequence $a_N$ such that for arbitrary $m\in \mathbb{N}$ and arbitrary subset $A\subset V_m$, given any two points $x,y \in A$ with $d(x,y) = N$ we have 
$$a_N(\frac{5}{3})^m \le c_{x,y}^A  $$ where $d(x,y)$ denotes the length of shortest path in $V_m$ from $x$ to $y$ without intersecting any other points in $A$.
\end{Theorem}

\begin{proof}
Let us define two sequences $\{a_n\},\{b_n\}$ with $a_1=1$ and $b_1=1$. The two sequences satisfy the recursive relations $$a_{i+1}=\frac{a_i}{(i+3)b_i}$$ and $$b_{i+1}=4\left(1+\frac{b_i}{a_i}\right)~.$$
For two given points $x,y\in A \subset V_m$, let $x\rightarrow x_1 \rightarrow x_2 \rightarrow \dots \rightarrow x_{N-1} \rightarrow y$ be a shortest path connecting them without intersecting any other points in $A$. Then this path satisfies all the conditions of the previous lemma. It is easy see that $a_i\left(\frac{5}{3}\right)^m \le m_i$ and $M_i \le \left(\frac{5}{3}\right)^m b_i$. This follows from the previous lemma and the recursive definition of $a_n$ and $b_n$.\\

So we have $a_N  (\frac{5}{3})^m \le c_{x,y}^{N-1}\le c_{x,y}^A$ as desired.
\end{proof}

\subsection{Specific Sets}
Now we turn to specific sets of interest in $V_m$.
\subsubsection{2-Set}
We define the following set as follows: $$\beta_0 = \{q_0, q_1, q_2, p_2 = F_2 F_1 q_0,p_1 =  F_1 F_0 q_2,p_0=  F_0 F_2 q_1 \},$$
$$\beta_m = \bigcup_{i}{F_i \beta_{m-1}}.$$

 \begin{figure}[ht]
\hspace{30 pt}
 \includegraphics[scale=0.5]{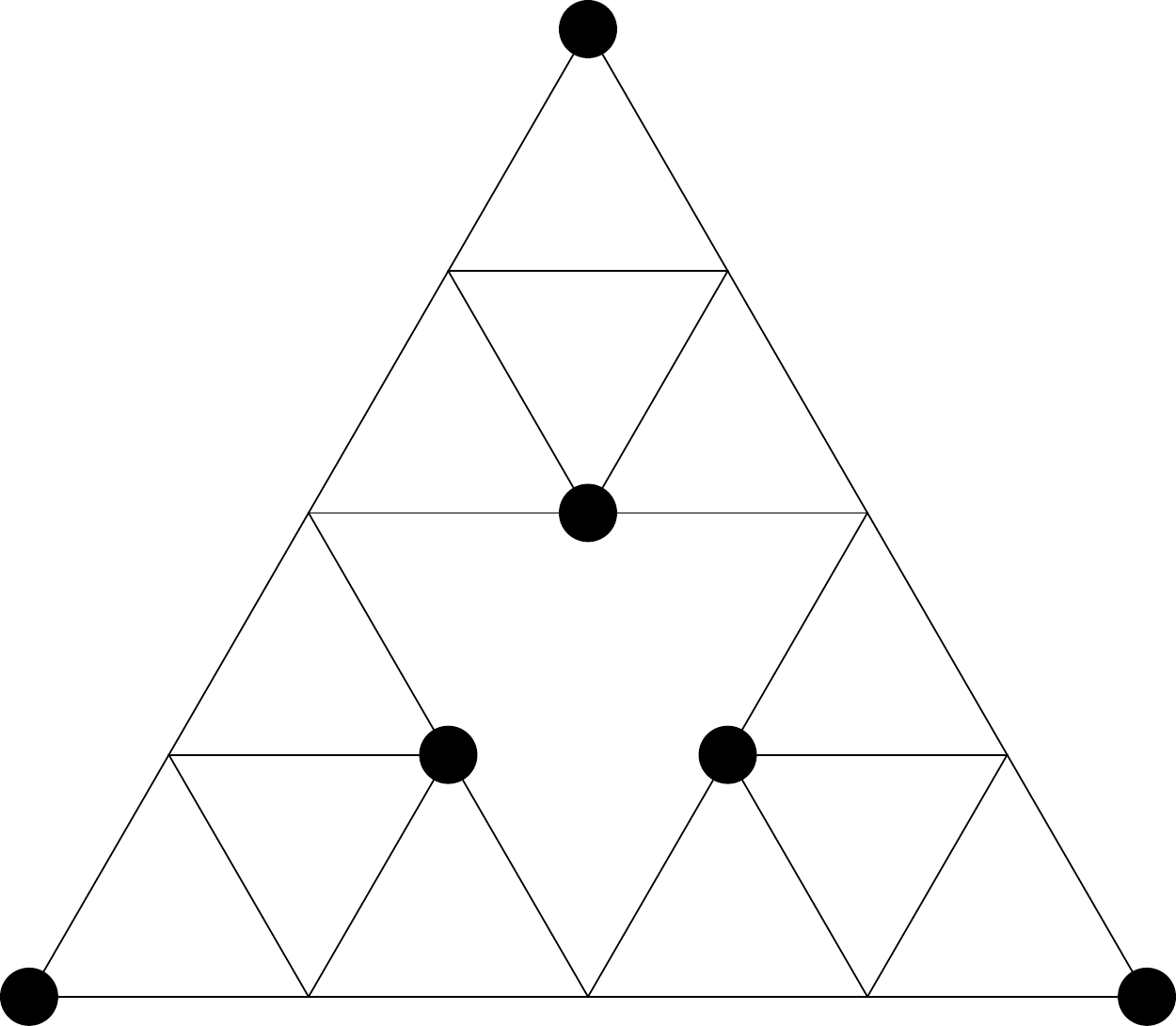}\hspace{30 pt} 
 \includegraphics[scale=0.5]{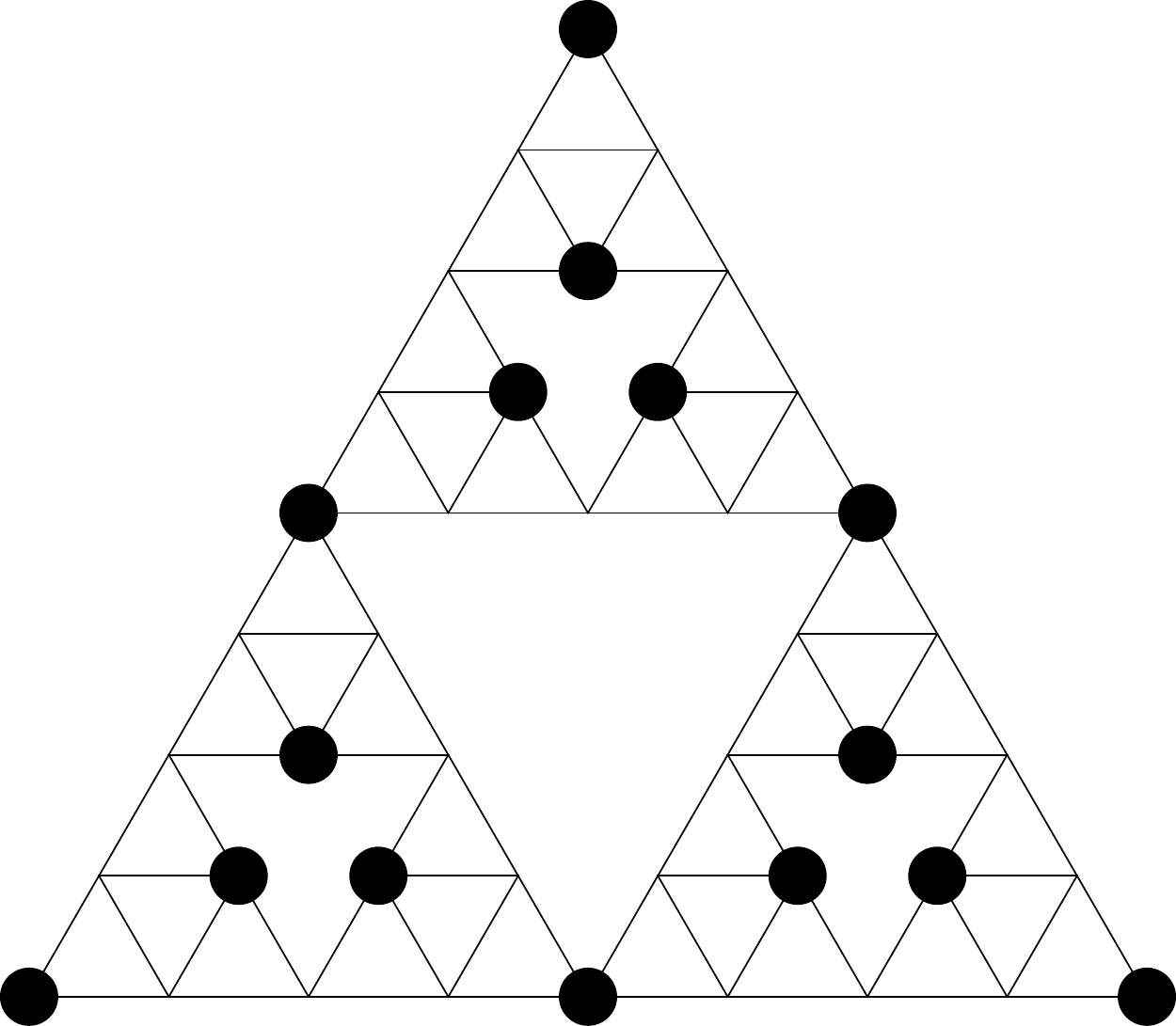} 
 \caption{$\beta_0$ and $\beta_1$}
 \label{figure:2set}
\end{figure}
A direct computation yields the following.
\begin{Lemma}
For $\beta_0$, the coefficients for $i \neq j$ are
 $$c_{q_i,p_i}=410/159~,$$ $$c_{q_i,q_{j}}=5/53~,$$ $$c_{q_i, p_{j}} = 20/53,$$ $$c_{p_i,p_{j}}=80/53~.$$

\end{Lemma}
\begin{Lemma}
We consider level $m>0$ and we let $\tilde{c}$ be the coefficients established for $\beta_0$. \\
If  $x$ and $y$ are in different $m$ cells, we have 
$$c_{x,y} = 0 ~.$$
If $x=F_w(\tilde{x})$ and $y=F_w(\tilde{y})$ for some $w$ such that $|w|=m$, then
$$c_{x,y}=(\frac{5}{3})^{m} \tilde{c}_{\tilde{x},\tilde{y}}~.$$
\end{Lemma}
\begin{proof}
When $x$ and $y$ are in different $m$-cells it is easy to see by Theorem \ref{zerocoefficients} that $c_{x,y} =  0$. For $m>0,$ $V_m \subset \beta_m$, we can use the electrical network model of conductances. We glue the network $F_w(SG)$ with the same graph and conductance and multiply the conductances in $\beta_0$ by $(5/3)^m$. Therefore, we obtain the stated conductances for higher levels.
\end{proof}

\subsubsection{New Level}

Here we consider $V_n \setminus V_{n-1}$. Let $\Gamma_n$ represent the graph representation of $V_n$ and let $d(\cdot, \cdot)$ denote the graph distance.  We denote by $c_{x,y}^n$ the conductance between $x$ and $y$ as elements in $\Gamma_n$. \\

 \begin{figure}[ht]
 
 \hspace{30 pt}
\raisebox{3ex}{ \includegraphics[scale=0.5]{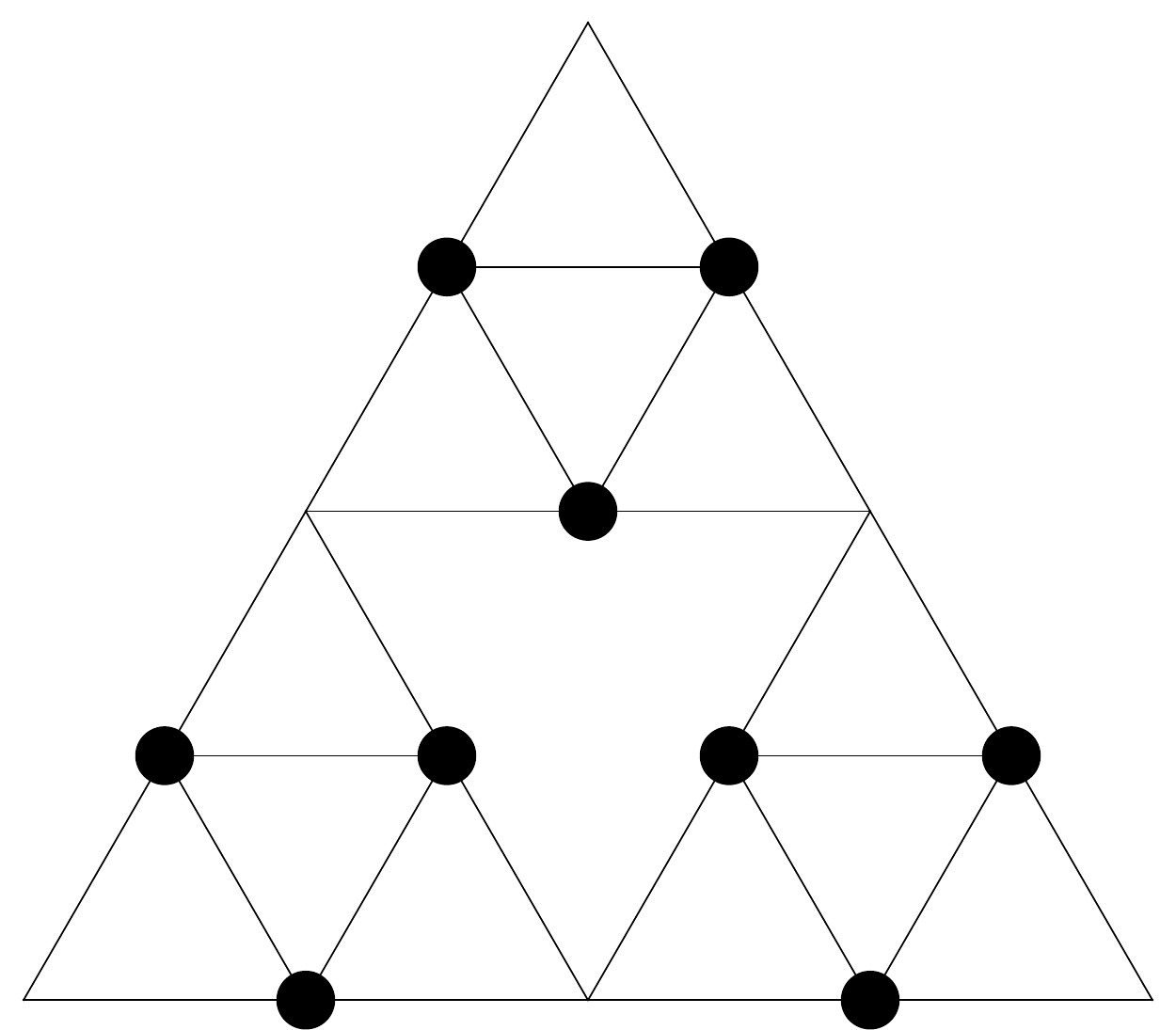} }\hspace{30 pt}
 \includegraphics[scale=0.45]{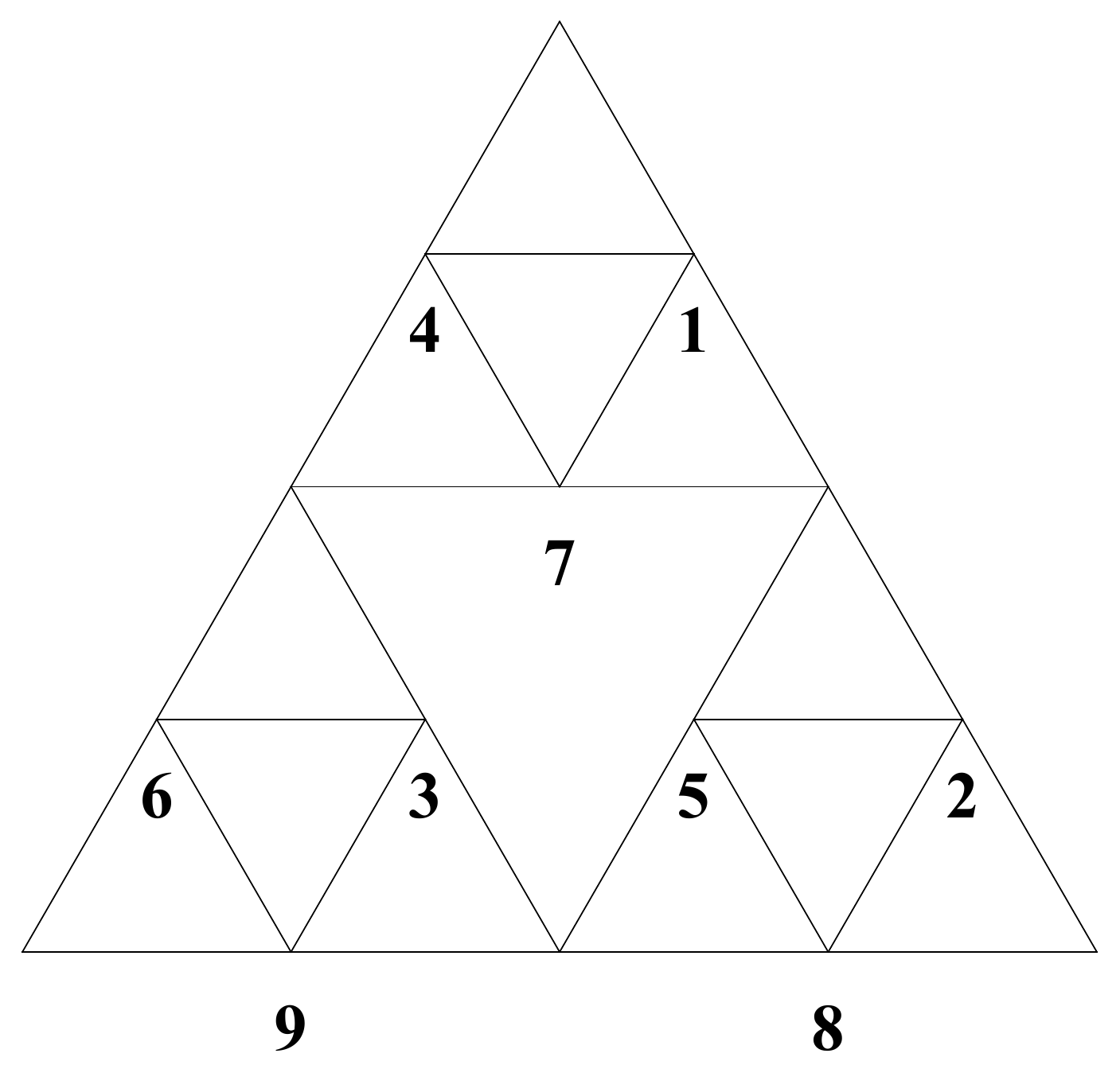} 
 
\caption{The New Level Set on Level 2 with an arbitrary ordering used to describe the coefficients}
\label{figure:newlevel}
\end{figure}

For each $x_i=\frac{q_{i-1}+q_{i+1}}{2}=F_{i-1}(q_{i+1}) \in V_1 \setminus V_0$ let $D^i_n$ denote the set of its four neighbors lying in $V_n$.

In this situation we cannot immediately extend from the first level. In fact we need the first 2 levels before we can use induction. We obtain the coefficients for the first 2 levels through direct computation and then extend.

For the first level, when we extend the values to $V_1 \setminus V_0$, all conductances are equal to $5/2$.

For the second level, as seen in Figure \ref{figure:newlevel}, from a simple calculation we have the following 3 types:

\begin{enumerate}
\item $c^2_{1,4} = c^2_{2,8} = c^2_{6,9} = 25/6$.
\item $c^2_{4,7} = c^2_{1,7} = c^2_{3,9} = c^2_{3,6} = c^2_{5,8} = c^2_{5,2} = 125/36$.
\item $c^2_{3,4}= c^2_{3,7}=c^2_{3,5}=c^2_{3,8}=c^2_{5,7}=c^2_{5,1}=c^2_{5,9}=c^2_{7,6}=c^2_{7,2}=c^2_{4,6}=c^2_{1,2}=c^2_{8,9}= 25/36$.
\end{enumerate}

We assume now that $n \geq 3$ and we obtain the following lemmas:

\begin{Lemma} \label{lemma:czero}
If $x$ and $y$ are two points such that $d(x,y) \geq 3$ then $c_{x,y}^n = 0$.
\end{Lemma}

\begin{proof}
This follows immediately from Theorem \ref{zerocoefficients}.
\end{proof}

 \begin{figure}[ht]
 \centering
 \includegraphics[scale=0.7]{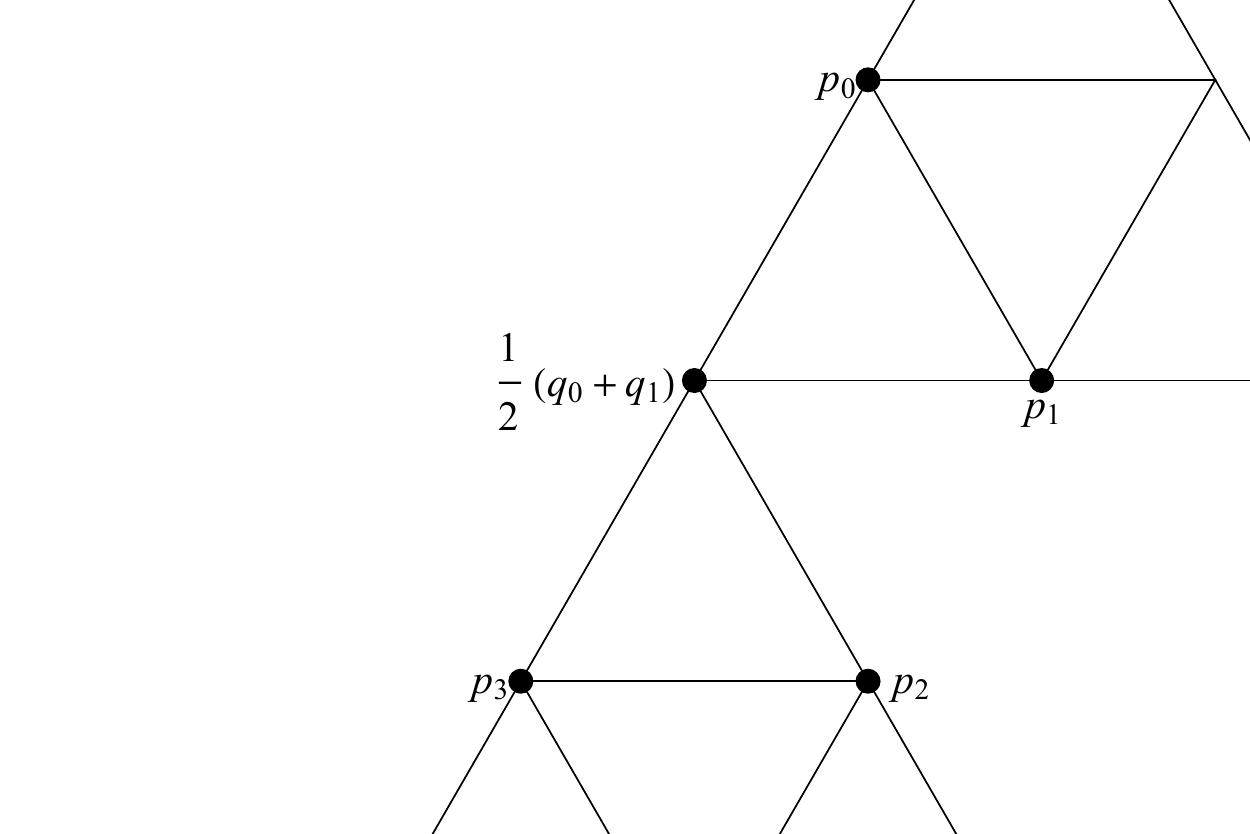}
 
\caption{The neighbors of a junction point. Here we see $\{p_j\} = D^i_n$}
\label{figure:junctionpoint}
\end{figure}
\begin{Lemma}(Four-neighbors) \label{lemma:fourneigh}
Let $a, b \in D_n^i$. If $d(a,b) = 1$ then $c_{a,b}^{n} = (\frac{5}{3})^n \frac{5}{4}$. In the case $d(a,b) = 2$, we have $c_{a,b}^{n} = (\frac{5}{3})^n \frac{1}{4}$.
\end{Lemma}

\begin{proof}
Follows from a direct calculation after observing figure \ref{figure:junctionpoint}.
\end{proof}

\begin{Lemma}
For $n \geq 3$, let $x,y \in F_i(V_{n-1}\setminus V_{n-2}) \setminus (D_n^{i+1} \cup D_n^{i-1})$.  Then, $c_{x,y}^{n} = \frac{5}{3}c_{x,y}^{n-1}$.
\end{Lemma}

\begin{proof}
Follows from the local nature of the calculation of energy due to Lemma \ref{lemma:czero}.
\end{proof}

\subsection{Bottom Row}
Here we consider the set $E$ which is the bottom row of $V_n$. We use the natural ordering on $E$ from left to right to write $x_0, \dots, x_{2^n}$ and denote the values attained at $x_i$ by $t_i$. Unlike the last two examples, this set has a very complicated structure. To develop a method to determine the coefficients we add the top point of $V_n$ to $E$ and denote the value at this point $q_0$ by $a$. This allows us to establish a recurrence between coefficients. We then relate these coefficients to our original problem without the top point. \\

After determining the coefficients for the quadratic form, we will see that the energy of the minimizer has a unique structure expressible in terms of Haar functions that provides a discrete analog of the continuous result established in \cite{OS}.

\subsubsection{Bottom Row with Top Point}
Throughout this section we denote the coefficients for level $n$ as $c^n_{x,y}.$ Also we use the shorthand $c^n_{x_i,x_j} = c^n_{i,j}$.
We get the following relation between the energy of the minimizers on different levels:

\begin{equation} \label{equation: bottomlinerecurse}
\E_{n+1}(u) = \frac{5}{3}\left((a-b)^2 + (b-c)^2 + (c-a)^2 + \E_n(u \circ F_1) + \E_n(u \circ F_2)\right),
\end{equation}
where $b=u(\frac{1}{2}(q_0+q_1))$ and $c=u(\frac{1}{2}(q_0+q_2))$.
Now we want to minimize the energy with respect to $b$ and $c$ to get the quadratic form in terms of $a$ and $x_i$.

\begin{Lemma}\label{lemma:bottomlinetoppointrecurse}

We set $a_n = c^n_{q_0, x_0} = c^n_{q_0,x_{2^n }}$. Then 
$$c^n_{q_0, x_j} = 2 a_n \text{ for } 0 < j < 2^n$$
and $$a_n = \frac{7\cdot 5^n}{3^n + 6\cdot 10^n}~.$$

The energy is minimized with 

\begin{align*}
b  & = \frac{(3+2^{n+1}a_n)a + (2+2^{n+1}a_n)(a_n)( t_0+ 2\sum_{i=1}^{2^n-1}{t_i} + t_{2^n})}{(1+ 2^{n+1} a_n)(3+ 2^{n+1} a_n)}\\ 
&~~~~ +\frac{a_n( t_{2^n}+ 2\sum_{i=2^n+1}^{2^{n+1}-1}t_i + t_{2^{n+1}})}{(1+ 2^{n+1} a_n)(3+ 2^{n+1} a_n)} \\
 c  & =\frac{(3+2^{n+1}a_n)a + a_n( t_0+ 2\sum_{i=1}^{2^n-1}{t_i} + t_{2^n})}{(1+ 2^{n+1} a_n)(3+ 2^{n+1} a_n)} \\
 &~~~~ +\frac{ (2+2^{n+1}a_n)(a_n)	( t_{2^n}+ 2\sum_{i=2^n+1}^{2^{n+1}-1}t_i + t_{2^{n+1}})}{(1+ 2^{n+1} a_n)(3+ 2^{n+1} a_n)}~.
 \end{align*}

\end{Lemma}
\begin{proof}
We use induction and obtain the base case by direct computation. We omit this computation. \\

To use induction, we compute the derivative with respect to $b$ to obtain
\[
\frac{\partial \E_{n+1}}{\partial b} = \frac{5}{3}(2(b-a) + 2(b-c)) + \frac{5}{3} \frac{\partial \E_n}{ \partial b}.
\]
From the inductive hypothesis we have
\[
\frac{\partial \E_n}{ \partial b} = 2 a_n ((b-t_0) + (b-t_{2^n})) + 4 a_n(\sum_{i=1}^{2^n-1} (b-t_i)).
\]

Then, we obtain a system of linear equations from $\frac{\partial \E_{n+1}}{\partial b} = 0,
\frac{\partial \E_{n+1}}{\partial c} = 0:
$

$$
\left(\begin{matrix}
 2+2^{n+1}a_n&-1 \\ 
 -1& 2+2^{n+1}a_n
\end{matrix}\right)
\left(\begin{matrix}
b\\
c
\end{matrix}\right)
=\left(\begin{matrix}
a + a_n( t_0+ 2\sum_{i=1}^{2^n-1}{t_i} + t_{2^n }) \\
	a + a_n( t_{2^n}+ 2\sum_{i=2^n+1}^{2^{n+1}-1}t_i + t_{2^{n+1}})
\end{matrix}\right).
$$
We solve the above linear system to obtain the desired $b$ and $c$. \\

Upon examination we see $-2 c_{q_0, x_0} = \partial_{at_0} \E_{n+1} = \frac{1}{2} \partial_{at_j} \E_{n+1} =   -c_{q_0, x_j}$ for $0<j < 2^n$.\\
To calculate the explicit formula for $a_n$, we get $a_0=1$ and the following recurrence:
$$2 a_{n+1} = \partial_{at_0} \E_{n+1} = \frac{-10 a_n}{3(2^{n+1}a_n+1)}.$$
Solving this recurrence with our appropriate inital starting condition, we get
$$a_n = \frac{7\cdot 5^n}{3^n + 6\cdot 10^n}.$$

\end{proof}
By examining (\ref{equation: bottomlinerecurse}) we derive several facts about the coefficients, using 
$$c^{n+1}_{y,z} = \frac{-1}{2}\partial_{u(y),u(z)} \E_{n+1}.$$

\begin{Lemma}
We set $b_n = c^n_{0, 2^n}$. Then 
\begin{align*}
c^n_{0, j} = 2 b_n && \text{for}~~~~2^{n-1}  < j < 2^n, \\
c^n_{j, 2^n} = 2 b_n &&  \text{for}~~~~0 < j < 2^{n-1}  \\
c^n_{i, j} = 4 b_n &&  \text{for}~~~~0 < i < 2^{n-1}, 2^{n-1} < j < 2^n. 
\end{align*}
Also, we have 
$$b_{n} = \frac{49\cdot 25^{ 
  n }}{(5 \cdot 3^{n} + 16 \cdot 10^{n}) (3^n + 6 \cdot 10^{n})}.$$

\end{Lemma}

\begin{Lemma} \label{lemma:proved}
We set $l_n = c^n_{0, 1}$. Then we get the  recurrence
$$l_1 = \frac{20}{9},$$ 
$$l_{n+1} = \frac{5}{3} l_{n} + \frac{98 \cdot 25^{n+1} (10 \cdot 3^{n+1} + 39 \cdot 10^{n+1})}{(3^{n+1} + 6 \cdot 10^{n+1}) (5 \cdot 3^{n+1} + 
  16 \cdot 10^{n+1}) (5 \cdot 3^{n+1} + 9 \cdot 10^{n+1})}$$ for $n\ge 1$.
\end{Lemma}

\begin{Lemma}
We set $m_n = c^n_{{2^{n-1} -1}, {2^{n-1}}} = c^n_{{2^{n-1}}, {2^{n-1}}+1}$. Then we get the following formula in terms of $l_n$:
$$m_1 = l_1 = \frac{20}{9},$$ 
$$m_{n+1} = \frac{5}{3} l_{n} + \frac{294\cdot 25^{n+1}}{(3^{n+1} + 6 \cdot 10^{n+1}) (5 \cdot 3^{n+1} + 9 \cdot 10^{n+1})}$$ for $n \ge 1$.
\end{Lemma}

\begin{Lemma}
Given $0 < i < j < 2^{n}$ we have
$$c^{n+1}_{i,j} = \frac{5}{3}c^{n}_{i,j} + 
\frac{196\cdot 25^{n+1} (10 \cdot 3^{n+1} + 39\cdot 10^{n+1})}{(3^{n+1} + 6 \cdot 10^{n+1}) (5\cdot 3^{n+1} + 16 \cdot 10^{n+1}) (5 \cdot 3^{n+1} + 9\cdot 10^{n+1})}$$ for $n \ge 1$.
\end{Lemma}

\begin{Remark}
It is easy to see that if we fix two indices $i, j$ then the asymptotic behavior of $c^n_{i,j}$ is $(\frac{5}{3})^n$. If we look at the relative position of points, instead, with $0 \leq i < 2^{n-1}$ and $ 2^{n-1} < j \leq 2^n$ we see the asymptotic behavior of $c^n_{i,j}$ is $(\frac{1}{4})^n$.  

\end{Remark}

\subsubsection{Bottom Row without Top Point}
Here we let $\tilde{c}^n_{i,j}$ be the coefficient of $(u(x_i) - u(x_j))^2$ in the minimizing form. We obtain analogous formulas for this case by noting that we minimize the energy by setting 
$$a = \frac{t_0 + 2(t_1 + \dots + t_{2^n - 1}) + t_{2^n}}{2^{n+1}}.$$
\begin{Lemma}
We have
$$\tilde{c}^n_{0,2^n} = c^n_{0,2^n} + \frac{a_n}{2^{n+1}},$$
\begin{displaymath}
\left\{ \begin{array}{lr}
\tilde{c}^n_{0,j} = c^n_{0,j} + \frac{a_n}{2^{n}} & 0 < j < 2^n,\\
\tilde{c}^n_{i,j} = c^n_{i,j} + \frac{a_n}{2^{n-1}} & 0 < i  < j < 2^n.
\end{array} \right.
\end{displaymath}
\end{Lemma}
\begin{Corollary}\label{corollary:farpointcoefficient}
We set $\tilde{b}_n = \tilde{c}_{0,2^n}^{n}$. Then
\begin{align*}
\tilde{c}_{0, j}^n = 2 \tilde{b}_n && 2^{n-1} < j < 2^n, \\
\tilde{c}_{j, {2^n}}^n = 2 \tilde{b}_n &&  0 < j < 2^{n-1}, \\ 
\tilde{c}_{i, j}^n = 4 \tilde{b}_n && 0 < i < 2^{n-1}, 2^{n-1} < j < 2^n .
\end{align*}
Also we have 
$$\tilde{b}_n = \frac{a_n}{2^{n+1}} + b_n.$$
More specifically,
$$\tilde{b}_n = \frac{35 \cdot 5^n}{10 \cdot 6^{n} + 32 \cdot 20^{n}}.$$

\end{Corollary}

\begin{Remark}
Since $a_n$ grows at a rate of $(\frac{1}{2})^n$, the previous lemma shows  that we add only a term of order $(\frac{1}{4 })^n$ to the coefficients. Since all the asymptotics from the case including the top point were at least this, we get identical asymptotics.
\end{Remark}

\subsubsection{Haar Structure}
Following \cite{OS}, we prove a discrete equivalent of those results showing that the energy of the minimizer can be expressed in a specific form using Haar functions. We first introduce the notation used in that paper.

We define the  Haar functions as follows
\begin{displaymath}
   \Psi_{n,k}(t) = \left\{
     \begin{array}{lr}
       2^{n/2} &  t \in [\frac{k}{2^n},\frac{k+1/2}{2^{n}})\\
       -2^{n/2} &   t \in (\frac{k+1/2}{2^n},\frac{k+1}{2^{n}}]\\
       0		&  ~~\text{otherwise}~.
     \end{array}
   \right.
\end{displaymath} 

Given a function $u$ with values $u(x_i) = t_i$ on the bottom row of $V_m$ we define a function $f_m(u)$ on the bottom line which is a piecewise constant interpolation:
\begin{equation} \label{equ:fm}
f_m(u) = t_0 \chi_{[0,\frac{1}{2^{m+1}}]} + \sum_{i=1}^{2^m - 1}{t_i \chi_{[\frac{2i - 1}{2^{m+1}},\frac{2i + 1}{2^{m+1}}]}} + t_{2^m}\chi_{[\frac{2^{m+1} - 1}{2^{m+1}},1]}.
\end{equation}
Using the $L^2$ inner product we define 
\begin{equation}\label{equ:inner}
D^m_{n,k}(u) = <\Psi_{n,k},f_m(u)>.
\end{equation}

When there is no ambiguity about which function we are using we use $D^m_{n,k}= D^m_{n,k}(u)$. These give us weighted averages of dyadic groups of values on the bottom row. We give some examples below:
$$D^m_{0,0} = (2^{-\frac{0}{2}-1})\left(\frac{(t_0 + 2 t_1 + \dots + 2 t_{2^{m-1} - 1} + t_{2^{m-1}}) - (t_{2^{m-1}} + 2 t_{2^{m-1} + 1} + \dots + 2 t_{2^{m} -1} + t_{2^m})}{2^m}\right),$$

$$D^m_{m-1,0} = (2^{-\frac{m-1}{2}-1})\left(\frac{(t_0  + t_1) - (t_1 + t_2)}{2}\right),~~{D}^m_{m-1,1} = (2^{-\frac{m-1}{2}-1})\left(\frac{(t_2 + t_3) - (t_3 + t_4)}{2}\right),$$
$$D^m_{m,0} =(2^{-\frac{m}{2}-1}) (t_0 - t_1)~~\text{and}~~D^m_{m,1} = (2^{-\frac{m}{2}-1})(t_1 - t_2).$$
We prove the following about the energy of the minimizer:

\begin{Theorem}\label{theorem:haarstructure}

\begin{align*}
\E(u) & = \sum_{i=0}^{m}\gamma^m_i\sum_{j=0}^{2^{m-i} - 1}{(D^m_{m-i,j})^2},
\end{align*}
where $$\gamma_{0}^1 = 4,$$ $$~~~~~~~~\gamma^m_n = \frac{10}{3}\gamma^{m-1}_n~\text{ for }0 \leq n < m ~\text{and}~ m\geq 2,$$  $$\gamma^m_m = \frac{70\cdot 10^{m}}{5 \cdot 3^m + 16\cdot 10^m}~~\text{for}~ m \geq 1.$$
\end{Theorem}

\begin{Remark}
We see that $\gamma_i^m \approx \left(\frac{10}{3}\right)^{m-i}$ and $\lim_{m\to\infty} \gamma_{m}^m=\frac{35}{8}.$ This is a discrete analog of Theorem 2.4 of \cite{OS}.
\end{Remark}
We start off showing some other properties of the minimizer. We begin by creating a sort of bump function which allows us to vary our functions harmonically away from the bottom row and the top point  while keeping our initial conditions on bottom row. When we refer to the bottom row throughout this section, we mean the subset of $V_m$ on the bottom row for the appropriate $V_m$ in the context.\\

\begin{Lemma}\label{lemma:integralharmonicextension} We define $\Phi_m$ to be the unique function for which $\Phi_m = 0$ on the bottom row of $V_m$, $\Phi_m(q_0) = 1$ and $\Phi_m$ is harmonic away from the bottom row of $V_m$ and the top point. Set 
$$~~~~~~~~b_m = \partial_n \Phi_m(q_0)~~\text{and}$$ 
$$c_m = \Phi_m(\frac{q_0 + q_1}{2}) = \Phi_m(\frac{q_0 + q_2}{2}).$$

Then for $m \ge 1$ we have 
$$~~~~~~~~b_m = \frac{14\cdot 10^m}{3^m + 6\cdot 10^m}~~\text{and}$$
$$c_m = \frac{5\cdot 3^m + 9 \cdot 10^m}{5\cdot 3^m + 30 \cdot 10^m}.$$
\end{Lemma}
\begin{proof}
Because $\Phi_m \circ F_0$ is harmonic we have $$b_m = \frac{5}{3}(2 - 2 c_m).$$
Furthermore we can combine the matching normal derivative conditions at $\frac{q_0 + q_1}{2}$ and the fact that $\Phi_m \circ F_1 = c_m \Phi_{m-1}$ to get
$$(\frac{5}{3})^2(2 c_m - 2 c_m c_{m-1}) + (\frac{5}{3})(2c_m - c_m - 1) = 0.$$
Thus we get the recurrence
$$c_m =\frac{-3}{-13 + 10 c_{m-1}}$$ and since $c_1=\frac{1}{3}$ so
 we can solve to get the stated closed form for both $c_m$ and $b_m$.
\end{proof}

Now we find the exact weighted average of the values on the bottom row and discover some other interesting properties of the minimizer.

\begin{Lemma} \label{Lemma:harmonicbottom}
Let $u$ be a function harmonic away from the bottom row of $V_m$. Then 
$\partial _nu(q_0)=0, u(q_0) = \int_{SG}{u d\mu}$ and
\begin{equation}\label{equation:bottomlineintegral}
\int_{\text{SG}}u d\mu=\frac{1}{2^{m+1}}\left(u(x_0)+2u(x_1) + \dots + 2u(x_{2^m - 1}) + u(x_{2^m})\right)=u(q_0).
\end{equation}

\end{Lemma}

\begin{proof}
We prove this by induction. For the base case we look at $V_0$ and we see that $$u(q_0) = \frac{1}{2}(u(q_1) + u(q_2))$$ which implies that $\partial_nu(q_0) = 0$ and (\ref{equation:bottomlineintegral}) is satisfied trivially. \\

Now suppose we have the result for $V_{m-1}$. We use the notation $u(x_i) = t_i$ for the values on the bottom row. Let $u_1$ be the function on $V_{m-1}$ with prescribed values $u_1(x_i) = t_i$ for $0 \leq i \leq 2^{m-1}$ with $\partial_n u_1(q_0) = 0$ which is harmonic away from the bottom row of $V_{m-1}$. Similarly let $u_2$ be the function with $u_2(x_i) = t_{2^{m-1} + i}$ for $0 \leq i \leq 2^{m-1}$ which also satisfies $\partial_n u_2(q_0) = 0$ and is harmonic away from the bottom row of $V_{m-1}$. We set 
$$a = \int_{SG} u_1 d \mu,$$
$$b = \int_{SG} u_2 d \mu$$
and define a function $u$ piecewise on $V_1$ by:
$$u \circ F_1 = u_1 + x \Phi_{m-1},$$
$$u \circ F_2 = u_2 + y \Phi_{m-1},$$ 
\begin{equation} \label{equation:uf0}
u \circ F_0 = \left(\frac{a+b+x+y}{2}\right) h_0 + (a+ x) h_1 + (b+y) h_2.
\end{equation}
In order for this to be harmonic away from the bottom row, we must have the proper gluing conditions. Thus we set out to pick $x,y$ to meet these conditions. From the inductive hypothesis we see we get continuity at $\frac{q_0 + q_1}{2}$ and $\frac{q_0 + q_2}{2}$ by
$$~~~~~~~~u \circ F_1(q_0) = u_1(q_0) + x = \int_{SG}{u_1 d \mu} + x = v \circ F_0(q_1)~~\text{and}$$ 
$$u \circ F_2(q_0) = u_2(q_0) + y = \int_{SG}{u_2 d \mu} + y = v \circ F_0(q_2).$$
All that remains to show is matching normal derivatives:
$$~~~~~~~~\partial_n (u \circ F_1)(q_0) + \partial_n (u \circ F_0) (q_1) = 0 ~~\text{and}$$ 
$$\partial_n (u \circ F_2)(q_0) + \partial_n (u \circ F_0) (q_2) = 0 .$$
This gives us a system of linear equations. Solving it we get
\begin{equation}\label{equation:matrix}
\left(\begin{matrix} b_{m-1} + \frac{3}{2} & -\frac{3}{2} \\
-\frac{3}{2} & b_{m-1} + \frac{3}{2} \end{matrix} \right) \left( \begin{matrix} x \\ y \end{matrix} \right) = \frac{3}{2}(a-b) \left(\begin{matrix} 1 \\ -1 \end{matrix} \right).
\end{equation}
From Lemma \ref{lemma:integralharmonicextension} we know $b_{m-1} > 0$ , so we can solve this system. Specifically, we have $x+y = 0$ and $x= C_m(a-b)$, $y=-C_m(a-b)$ for some constant $C_m$.
Thus we have found a unique function which is harmonic away from the bottom row with specified values. Note from our set up that $$\partial_n u(q_0) = \frac{5}{3}\left(2(u \circ F_0)(q_0) - (u \circ F_0)(q_1) - (u \circ F_0)(q_2)\right) = 0.$$ 
Now we see
$$\int_{SG}{u d\mu} = \frac{1}{3}\sum_i{\int_{SG}{u \circ F_i d \mu}} $$
$$ =\frac{1}{3}(a + b + \frac{1}{3}(a+x + b+ y+ \frac{a+b+x+y}{2}))$$
$$ = \frac{a+b}{2} = u(q_0).$$
Thus we have all the stated properties.

\end{proof}

\begin{Remark}
One can show in fact $F_0^n(\frac{q_1 + q_2}{2}) = \int u d \mu$ for all $n >0$. Since we see from above $x = -y = C_m(a-b)$ for some constant $C_m$, we also get $\E(u \circ F_0) = \tilde{C}_m (a-b)^2$ for some $\tilde{C}$.
\end{Remark}

Using the same set up as the previous lemma, we obtain an inductive argument to prove Theorem \ref{theorem:haarstructure}.
\begin{proof}[Proof of Theorem \ref{theorem:haarstructure}]

For $m=1$, namely for $V_1$,  we obtain
\[
\E(u) = C_1 \left( (t_0 - t_1)^2 + (t_1 - t_2)^2\right) + C_2(t_0 - t_2)^2.
\]
A direct calculation shows that the energy of the function that harmonically extends the values $t_0, t_1, t_2$ on the bottom row of $V_1$ has $C_1 = \frac{5}{2}$ and $C_2 = \frac{1}{4}$,
where 

\begin{align*}
(t_1 - t_2)^2 &= 8 (D_{1,1}^1)^2 \\
(t_0 - t_1)^2 &= 8 (D_{1,0}^1)^2 \\
(t_0 - t_1)^2 &= 16 (D_{0,0}^1)^2.
\end{align*}

Therefore in this case, we obtain $\gamma_0^1 = 4, \gamma_{1}^1 =  20$.
Now, we suppose it holds in the case $V_{m-1}$ and prove it for $V_{m}$:
From the self similar identity and the proof of Lemma \ref{Lemma:harmonicbottom}, we get
\begin{align} \label{equ:Eu}
\E(u) &= \frac{5}{3}\left(\E(u \circ F_0) + \E(u \circ F_1) + \E(u \circ F_2)\right) \nonumber \\
&=\frac{5}{3}\left(\E(u \circ F_0) + \E(u_1 + x \Phi_{m-1}) + \E(u_2 + y \Phi_{m-1}) \right) \nonumber \\
&=\frac{5}{3} \left( \E(u_1) + \E(u_2) + \E(u \circ F_0) + 2 x \E(u_1,\Phi_{m-1}) \right. \nonumber\\
&\left. + x^2 \E(\Phi_{m-1}) + 2 y \E(u_2,\Phi_{m-1}) + y^2 \E(\Phi_{m-1}) \right).
\end{align}

From the induction hypothesis, we have 
$\E(u_l) = \sum_{i=0}^{m-1}{\gamma}^{m-1}_i\sum_{j=0}^{2^{m-i} - 1}{ {D}^{m-1}_{m-1-i,j}(u_l)^2}$ for $l=1,2$. Also by equations (\ref{equ:fm}) and (\ref{equ:inner}) we have 
\begin{align}
D^{m-1}_{m-1-i,j}(u_1)&=2^{\frac{1}{2}}D^m_{m-i,j}(u) \label{equ1:D}\\
D^{m-1}_{m-1-i,j}(u_2)&=2^{\frac{1}{2}}D^m_{m-i,j+2^{m-1-i}}(u) \label{equ2:D}.
\end{align}

 We want to obtain $\gamma_{0,n}^m$ for $0 \leq n < m$. If we plug the equations (\ref{equ1:D}) and (\ref{equ2:D}) in $\E(u_1)+ \E(u_2)$ we obtain

\begin{align} \label{equ:E1plusE2}
\E(u_1)+ \E(u_2) &= \sum_{i=0}^{m-1}{\gamma}^{m-1}_i\sum_{j=0}^{2^{m-1-i} - 1}{{D}^{m-1}_{m-1-i,j}(u_1)^2}+\sum_{i=0}^{m-1}{\gamma}^{m-1}_i\sum_{j=0}^{2^{m-1-i} - 1}{{D}^{m-1}_{m-1-i,j}(u_2)^2}\\ &= 2\sum_{i=0}^{m-1}{\gamma}^{m-1}_i\sum_{j=0}^{2^{m-i} - 1}{{D}^m_{m-i,j}(u)^2}.
\end{align}
Now, what we want to show is the remaining part in equation (\ref{equ:Eu}) is a multiple of  $D^m_{0,0}(u)^2.$
We set $a = u_1(q_0)$ and $b = u_2(q_0)$. By the right side of equation (\ref{equation:bottomlineintegral}) we have
\begin{align*}
(2^{-2})(a-b)^2 &= (2^{-2})4^{-m}\big((t_0 + 2 t_1 + \dots + 2 t_{2^{m-1} - 1} + t_{2^{m-1}}) - (t_{2^{m-1}} + 2 t_{2^{m-1} + 1} + \dots + 2 t_{2^{m} -1} + t_{2^m})\big)^2\\
&=D^m_{0,0}(u)^2.
\end{align*}
Thus we wish to show first that
$$\frac{5}{3}\left(\E(u \circ F_0) + 2 x \E(u_1,\Phi_{m-1}) + 2 y \E(u_2,\Phi_{m-1}) + (x^2 + y^2) \E(\Phi_m)\right) = k(a-b)^2$$
for some constant $k$ which only depends on $m$. \\

From equations (\ref{equation:uf0}) and (\ref{equation:matrix}) , we have that $\E(u \circ F_0) = \tilde{C}_m (a-b)^2$ and also $x^2 + y^2 = 2 C_m^2(a-b)^2$ where $C_m,\tilde{C_m}$ only depend on $m$. Thus all that remains is to show 
\begin{align*}
2 x \E(u_1,\Phi_{m-1}) + 2 y \E(u_2,\Phi_{m-1}) &= 2 C_m (a-b) \left( \E(u_1,\Phi_{m-1}) - \E(u_2,\Phi_{m-1}) \right) \\
\end{align*} is a multiple of $(a-b)^2$.
By the polarization identity, we get
\begin{align*}
\E(u_1,\Phi_{m-1}) &= \frac{1}{4}\left(\E(u_1 + \Phi_{m-1}) - \E(u_1 - \Phi_{m-1})\right). \\
\end{align*}
For $l=1,2$, $u_l \pm \Phi_{m-1}$  are harmonic away from the bottom row and the top point. They agree on the bottom row so we have a lot of cancellation. We let $a_i$ denote the coefficients from Lemma \ref{lemma:bottomlinetoppointrecurse}. Then,
\begin{align*}
4 \E(u_1,\Phi_{m-1}) &= \left(\E(u_1 + \Phi_{m-1}) - \E(u_1 - \Phi_{m-1})\right) \\
&=a_{m-1}(a+1 - t_0)^2 + 2a_{m-1} \left(\sum_{i=1}^{2^{m-1} -1}(a+1 - t_i)^2\right) +a_{m-1}(a+1 - t_{2^{m-1}})^2 \\
&-\left(a_{m-1}(a-1 - t_0)^2 + 2a_{m-1}\left( \sum_{i=1}^{2^{m-1} -1}(a-1 - t_i)^2\right) +a_{m-1}(a-1 - t_{2^{m-1}})^2\right)\\
&=4a_{m-1}(a - t_0) + 8a_{m-1} \left(\sum_{i=1}^{2^{m-1} -1}(a - t_i)\right) + 4a_{m-1}(a - t_{2^{m-1}}) \\
&= 2^n a_{m-1} a - 4a_{m-1}\left(t_0 + 2 \left(\sum_{i=1}^{2^{m-1} -1}t_i\right) + t_{2^{m-1}}\right)\\
&= (2^m a_{m-1} - 4 a_{m-1})a.
\end{align*}
Similarly, we get 
$$4 \E(u_2,\Phi_{m-1}) = (2^m a_{m-1} - 4 a_{m-1})b.$$
Thus we have 
\begin{align*}
2 x \E(u_1,\Phi_{m-1}) + 2 y \E(u_2,\Phi_{m-1}) &= 2 C_m (a-b) \left( \E(u_1,\Phi_m) - \E(u_2,\Phi_m) \right) \\
&=\tilde{k} (a-b)^2.
\end{align*}
So, we can write
\begin{equation}
\E(u) = \frac{5}{3} \left( 2\sum_{i=0}^{m-1}{\gamma}^{m-1}_i\sum_{j=0}^{2^{m-i} - 1}{{D}^m_{m-i,j}(u)^2} + 4 \tilde{k} D_{0,0}^m(u)^2 	\right)
\end{equation}
Since $D_{0,0}^m(u)^2$ is the only term that includes $t_0$ and $t_{2^m}$, we have $-2 c_{x_0, x_{2^m}} = \partial_{t_0, t_{2^m}} \E = \frac{-10}{3} 4^{-m} \tilde{k}$. We note that $4^{-m-1}\gamma^m_m = \frac{-1}{2}\partial_{t_0,t_{2^m}}\E$ and from Corollary \ref{corollary:farpointcoefficient} we get the stated closed form.
\end{proof}

\section{Laplacian}
Given a positive, finite measure $\zeta$ on $SG$ we now define a functional $T_\zeta(u) =  \int (\Delta_\zeta u)^2 d \zeta$ on the space $\dom_{L^2}\Delta_\zeta.$ We first consider the constraints $u(x_i) = a_i$ and $\partial_n u(y_j) = b_j$ for all $x_i \in E$ and $y_i \in F$, with $E,F$ finite subsets of $V_m$. We need the following assumption on $E$:
\begin{Assumption}
There exist 3 points $x_1, x_2, x_3 \in E$ such that $$H=\begin{pmatrix}
 h_0(x_1)&h_1(x_1)  &h_2(x_1) \\ 
 h_0(x_2)&h_1(x_2)  &h_2(x_2)\\ 
 h_0(x_3)&h_1(x_3)  &h_2(x_3)
\end{pmatrix}$$ is invertible, where $h_j$ denotes the harmonic function with boundary values $h_j(q_i) = \delta_{ij}$.
\end{Assumption} 
We see from the theory of harmonic coordinates on the Sierpinski Gasket \cite{Ki03} that this is equivalent to guaranteeing $E$ is not contained in a line in the harmonic Sierpinski Gasket. In the situation where $E$ is contained on a line we can see that we should not expect uniqueness. For example, there is a one dimensional space of skew symmetric harmonic functions which vanish on the line through the middle of $SG$. See \cite{LS} for a complete discussion of this example.

\begin{Theorem}\label{theorem:laplacian1}
(Existence and Uniqueness) Set $$Y = \{u \in \dom_{L^2}\Delta_\zeta ~|~ u(x_i) = a_i\}.$$ Then we have a unique minimizer in $Y$ minimizing $T_\zeta$.
\end{Theorem}

\begin{proof}
Convexity holds trivially.
We consider the space
\[
\tilde{A} = \text{dom}_{L^2} \Delta_{\zeta} / H(K)
\]
where $H(K)$ is the space of harmonic functions with the norm $\|\cdot\| = \sqrt{\int_{SG}{|\Delta_\zeta u|^2 d \zeta}}$. Let us define $\tilde{Y} = Y / H(K)$. We want to show that $\tilde{Y}$ is closed.\\

We first consider the situation in which we have Dirichlet boundary conditions and then observe we can easily extend to the general situation. We denote the space $\text{dom}_{L^2}\Delta_0=\{u\in \text{dom}_{L^2}\Delta_\zeta:u|_{V_0}=0 \}$. Note that the space $\tilde{A} \cong \dom_{L^2}\Delta_0$. \\

In the case that $Y \subset dom_{L^2} \Delta_0$, we want to show that $Y$ is closed in $dom_{L^2}\Delta_0$
 with respect to the norm $|| u || :=\sqrt{ \int (\Delta_\zeta u)^2 d \zeta} $.  By $(\ref{equation:supenergy})$, we know that $||u||_{\infty}\le C\E(u)^{\frac{1}{2}}$ for some constant $C$.  By \cite{SU}*{Lemma 4.6}, we have
\[
\mathcal{E}(u) \leq C_0 || \Delta_\zeta u ||_2^2.
\]
Combining the two inequalities yields
\[
|| u ||_{\infty}^2 \leq C_1 || \Delta_\zeta u ||_2^2
\]
where $C,C_0,C_1$  are positive constants. Given $u_n \rightarrow u$ in $dom_{L^2}\Delta_0$ with $u_n$  in $Y$,  the above inequality implies that $ u(x_i) = a_i$ for any $i$ and $u$ vanishes on the boundary.  Therefore, $Y$ is closed.  \\

To extend the result, we view our space as $\tilde{A}$. Now suppose $\widetilde{u_n}\rightarrow \tilde{u}$, where $\widetilde{u_n}$ are in $\tilde{Y}$ and $u_n\in Y$ for all $n$. We note we can pick representatives with Dirichlet boundary conditions $w_n$ for each $u_n$. Since $H$ is invertible, we know that there exists a harmonic function $v$ such that $u(x_i)=a_i$ for $i=1,2,3$ where $u=w+v$ and $w$ has Dirichlet boundary conditions. Since $||w_n-w||_{\infty}\rightarrow 0$ and $u(x_i)=a_i$ we have $\lim_{n\rightarrow \infty} v_n(x_i)=v(x_i)$ for $i=1,2,3$. We note that using the invertibilty of $H$ we obtain $\lim_{n\rightarrow \infty} v_n(q_i)=v(q_i)$ for $i=0,1,2$. By the maximum principle for harmonic function, we know that $||v_n-v||_{\infty}\rightarrow 0$ as $n\rightarrow \infty$. As a result, we have $||u_n-u||_{\infty}\rightarrow 0$ as $n \to \infty$ so $u\in Y$ and $\tilde{u} \in \tilde{Y}$. Therefore, we again get that $\tilde{Y}$ is closed. 
\end{proof}
When we add in the normal derivatives constraint, we require a new condition on the measure,
\begin{equation}\label{equation:measure}
\|\Delta_\zeta u\|^2_2 \geq C_w \|\Delta_\zeta(u\circ F_w)\|^2_2 
\end{equation}
for every word $w$ and some constant $C_w$ depending only on $w$.
We now show that both the standard self-similar measure, $\mu$, and the Kusuoka measure, satisfy (\ref{equation:measure}).
\begin{Lemma}\label{lemma:measureproperty}
Both the self similar measure and the Kusuoka measure, $\nu$, satisfy (\ref{equation:measure}). Furthermore we can establish estimates on $C_w$.
\end{Lemma}
\begin{proof}
In the case of the self similar measure, it follows immediately from the self similar identity on the Laplacian which can be found in \cite{Str}. In this case, $C_w = r_w \mu_w^{-1} = \left(\frac{25}{3}\right)^{|w|}$.\\

We have the following "self-similar" identity for the Kusuoka measure (\cite{RCS}, Theorem 2.3),
\begin{align*} \label{equ2}
||\Delta_{\nu} u||_{2}^2 &= \int |\Delta_{\nu} u|^2 d \nu \\ &= \sum_{w}  \int Q_w |(\Delta_{\nu} u) \circ F_w|^2 d \nu = \sum_{w} \int \frac{1}{Q_w} |\Delta_{\nu}(u \circ F_w)|^2 d \nu \geq C_w ||\Delta_{\nu}(u \circ F_w)||_{2}^2,
\end{align*}
where $C_w$ is a constant depending on $w$ that we explore as follows.

We have $Q_j = \frac{1}{15}+ \frac{12}{25}\frac{d\nu_j}{d\nu}$ and $\nu_j$ and $Q_w$ are as defined in \cite{RCS}. By \cite{RCS}*{Theorem 3.5}, we have $ 0 \leq \frac{d\nu_j}{d\nu} \leq \frac{2}{3}$. Therefore, we obtain $\frac{1}{15} \leq Q_j \leq \frac{29}{75}$ in which case by \cite{RCS}*{Corollary 2.4}, for a general word $w$, we have
\[
\left( \frac{1}{15} \right)^{|w|} \leq Q_w \leq \left( \frac{29}{75} \right)^{|w|}
\]
so

\[\left( \frac{75}{29} \right)^{|w|} \leq C_w \leq 15^{|w|}.
\]
\end{proof}

Now we add in the normal derivative constraints and show we are guaranteed a minimizer. 
\begin{Lemma}\label{lemma:energyconvergence}
Let everything be as in Theorem \ref{theorem:laplacian1}, so $Y=\{u\in dom_{L^2}\Delta_\zeta :u(x_i)=a_i\}$.
Define $\tilde{Y}=Y/H(K)$. Suppose $\tilde{u}_n \rightarrow \tilde{u} \in \tilde{Y}$ and $u_n,u\in Y$. If we set $z_n = u_n - u$ then we have $\E(z_n)\rightarrow 0$ as $n \rightarrow \infty $.
\end{Lemma}

\begin{proof}
As in Theorem \ref{theorem:laplacian1}, we write $u_n=w_n+v_n$ and $u=w+v$ where $w_n, w$ are Dirichlet functions and $v, v_n$ are harmonic. 
Since $w_n$ and $w$ have Dirichlet boundary conditions, we can use \cite{SU}*{Lemma 4.6} to conclude that
\[
\E(w_m-w) \leq C || \Delta_\zeta (w_m-w) ||_{2}^2.
\] We can use this to get $\E(w_n - w)\rightarrow 0$ when $n\rightarrow \infty$. Also since $\lim_{n\rightarrow \infty} v_n(q_i)=v(q_i)$ for $i=0,1,2$ we have $\E(v_n-v)\rightarrow 0$ as $n\rightarrow \infty$. Hence, the result follows immediately from the triangle inequality.
\end{proof}
\begin{Theorem}\label{theorem:laplacian2}
(Existence and Uniqueness) Set $$X = \{u \in \dom_{L^2}\Delta_\zeta ~|~ u(x_i) = a_i, \partial_n u(y_j) = b_j\}.$$ Then for any measure $\zeta$ satisfying (\ref{equation:measure}) there is a unique minimizer in $X$ for $T_\zeta$.
\end{Theorem}

\begin{proof}
To extend the result, we are only concerned with showing that $\tilde{X}$ is closed as a subset of $\tilde{Y}$ (convexity is again obvious). Let $u_n \rightarrow u$ as $n \rightarrow \infty$ in $||\cdot||$ with $u_n \in Y$. We put $z_m = u_m - u$.  We pick $y_j = F_{w}q_i$ for some $w$ and $q_i$.  By using the Gauss-Green formula, we have
\[
\E (z_m \circ F_{w}, h_i) = - \int \Delta_\zeta (z_m \circ F_{w}) h_i d \zeta + \partial_n (z_m  \circ F_{w}) (q_i).
\]
So we obtain
\[
|\partial_n (z_m  \circ F_{w}) (q_i)| \leq |\E (z_m \circ F_{w}, h_i)| + |\int \Delta_\zeta (z_m \circ F_{w}) h_i d \zeta| \leq \tilde{C}(\mathcal{E}^{1/2}(z_m) + ||\Delta_\zeta z_m ||_2)
\]
for some constant $\tilde{C} >0$ from the self-similar identity for energy and (\ref{equation:measure}).
We apply Lemma \ref{lemma:energyconvergence} to see that as $m \to \infty$ we have $(\mathcal{E}^{1/2}(z_m) + ||\Delta_\zeta z_m ||_2) \to 0$ so $|\partial_n (z_m  \circ F_{w}) (q_i)| \to 0$. Thus taking the limit as $m \rightarrow \infty$ yields $\partial_n u(y_j) = b_j$.
\end{proof}

\begin{Remark}
Note that we may view this as a subset of the  value only constraints, and thus the results are fully generalized. We do not know what happens in the situation when $E$ does not satisfy our assumption and we impose boundary conditions.
\end{Remark}

\subsection{Construction}
In this section, we assume $E = F$ for simplicity and assume the measure $\zeta$ satisfies the appropriate conditions to guarantee the existence and uniquness of a minimizer.  We continue as before by finding Euler Lagrange equations for a necessary condition for the minimizer.\\

\begin{Lemma}\label{lemma:eulerlagrangelaplacian}
Given $u \in \dom_{L^2}\Delta_\zeta$ which minimizes $T_\zeta$, we must have 
\begin{equation}\label{eulerlagrangelaplacian}\int \Delta_\zeta u \Delta_\zeta v d \zeta = 0 \end{equation} for all $v \in \dom_{L^2}\Delta_\zeta$ with $v|_{E} \equiv 0$ and $\partial_n v|_{E} \equiv 0.$
\end{Lemma}
\begin{proof}
Suppose $u$ is the minimum.
Given any $v \in \dom_{L^2}\Delta_\zeta$ with $v|_{E} \equiv 0$ and $\partial_n v|_E \equiv 0$ we have that $u + tv$ also satisfies the constraints for any $t \in \R$. We consider $f(t) = T_{\zeta}(u+tv)$ and notice that since $u$ is a minimum by hypothesis we have $f'(0) = 0$. Since $$T_{\zeta}(u + tv) = \int (\Delta_\zeta u)^2 d \zeta + 2 t \int \Delta_\zeta u \Delta_\zeta v d \zeta + t^2 \int (\Delta_\zeta v)^2 d \zeta,$$
 we have
$$f'(0) = \int \Delta_\zeta u \Delta_\zeta v d \zeta = 0.$$
\end{proof}
We can apply this along with the Dirichlet boundary conditions to immediately get uniqueness:
\begin{Lemma} \label{lemma:laplacianuniqueness} (Uniqueness) If $u_1$ and $u_2$ both satisfy (\ref{eulerlagrangelaplacian}) then $u_1 = u_2$
\end{Lemma}
\begin{proof}
Specifically, since $u_1(x_i) = u_2(x_i) = a_i$ we get $(u_1 - u_2)|_{E} \equiv 0$ and likewise $\partial_n(u_1-u_2)|_E \equiv 0$. Therefore,  we can set $v = u_1 - u_2$ and use linearity of the integral and Laplacian to obtain
$$\int (\Delta_\zeta(u_1 - u_2))^2 d \zeta = 0.$$
This gives $\Delta_\zeta(u_1 - u_2) = 0$ almost everywhere, which implies $\E(u_1 - u_2) = 0.$
So  the functions differ by a harmonic function $h$. By our assumption that $H$ is invertible we immediately get that $h=0$ so $u_1 = u_2$.

\end{proof}
\begin{Lemma} \label{lemma: laplacianconstruction}
If $u$ is a piecewise biharmonic function on $m$-cells and $v \in \dom_{L^2}{\Delta_\zeta}$ then we have \begin{equation}\label{equation:biharmonic}\int \Delta_\zeta u \Delta_\zeta v d \zeta = \sum_{|w|=m}{r_w^{-1}\sum_{V_0}{\left( (\Delta_\zeta u \circ F_w )\partial_n(v \circ F_w) - (v \circ F_w)\partial_n (\Delta_\zeta u \circ F_w)\right)}}.\end{equation}	
\end{Lemma}
\begin{proof}
Let $v \in \dom_{L^2}{\Delta_\zeta}$. Then
$$\int \Delta_\zeta v \Delta_\zeta u d \zeta= \sum_{|w|=m}{\mu_w \int (\Delta_\zeta v \circ F_w)(\Delta_\zeta u \circ F_w)}.$$ Since $\Delta_\zeta u \circ F_w$ is harmonic, Gauss-Green implies this equals
$$\sum_{|w|=m}{\mu_w (\mu_w r_w)^{-1}\left(-\E(v \circ F_w, \Delta_\zeta u \circ F_w) +\sum_{V_0}{ (\Delta_\zeta u \circ F_w) \partial_n(v \circ F_w)}\right)}.$$
We can apply Gauss-Green again since $v \circ F_w \in \dom \E$ to obtain
$$ \sum_{|w|=m}{r_w^{-1}\left(\int (v \circ F_w) \Delta_\zeta(\Delta_\zeta u \circ F_w) d \zeta -\sum_{V_0}{(v \circ F_w)\partial_n (\Delta_\zeta u \circ F_w)} +\sum_{V_0}{(\Delta_\zeta u \circ F_w )\partial_n(v \circ F_w)}\right)}$$
and since $\Delta_\zeta u \circ F_w$ is harmonic we finally obtain the equation (\ref{equation:biharmonic}).
\end{proof}
\begin{Corollary}\label{corollary: biharmonicminimizer}
If $E=V_m$ then the minimizer must be piecewise biharmonic.
\end{Corollary}
\begin{proof}
Consider the piecewise biharmonic function $f$ that satisfies the constraints of Lemma \ref{lemma:eulerlagrangelaplacian}. This implies $f$ satisfies equation (\ref{eulerlagrangelaplacian}). Then, Lemma \ref{lemma:laplacianuniqueness} implies that it is the minimizer.

\end{proof}
\begin{Remark}
In the case where $E=V_m$ and $\zeta = \mu$, we can give a concrete way to calculate the energy. When $u$ is biharmonic, we have $\Delta u = \sum_{i=0}^{2} c_i h_i$ for some constants $c_i$. We know that a biharmonic function is completely determined by its values and the normal derivatives at $V_0$. We put $a_i = u(q_i),$ $b_i = \partial_n u(q_i),$ and $d_i = 2a_i - a_{i+1} - a_{i-1} - b_i$. We can interpret these $d_i$ as a measurement of how much the function differs from a harmonic function since all the $d_i$ are 0 if and only if the function is harmonic. By using equation (\ref{matrixeq}) we can solve $c_i$ in terms of $a_i$ and $b_i$. Reorganizing the equation in terms of $d_i$ yields
\begin{align*}
\Theta(a_i,b_i) := \int |\Delta u|^2 d \mu &=\int \left( \sum_{i=0}^{2} c_i h_i\right)^2 d \mu \\
&= 90 \sum_{i} a_i^2 + 
 30 \sum_{i\neq j}a_i b_j + 
 11 \sum_{i}b_i^2 - 90 \sum_{i < j} a_i a_j - 
 60 \sum_{i}a_i b_i - 8 \sum_{i < j}b_i b_j\\
 &=11 \sum_{i}d_i^2- 8 \sum_{i < j}d_i d_j.
 \end{align*}
Using the self similar identity we get for $V_m$:
$$\int|\Delta u|^2 d\mu = \left(\frac{25}{3}\right)^{m}\sum_{|w|=m} \Theta(F_w q_i,(\frac{5}{3})^m \partial_n F_w q_i).$$
\end{Remark}
~\\
~\\

We use Lemma \ref{lemma: laplacianconstruction}, to show that there is a unique function which satisfies the following properties, and that function is the minimizer. Namely the function is piecewise biharmonic on $V_m$ and its Laplacian is harmonic away from $E$. The exact conditions are the following,
\begin{enumerate}
\item
$u(x_i) = a_i,$
\item
$\partial_n u(x_i) = b_i,$
\item
$u \text{ is piecewise biharmonic on } V_m,$
\item
$\Delta_\zeta u \text{ is continuous at } y \text{ for } y \in V_m \setminus E,$
\item
$\Delta_\zeta u \text{ has matching normal derivatives at } y \text{ for } y \in V_m \setminus E.$
\end{enumerate}
From Lemma \ref{lemma: laplacianconstruction} we get that given a function $u$ which satisfies (1-5)  must satisfy (\ref{eulerlagrangelaplacian}) and from Lemma \ref{lemma:laplacianuniqueness} we have that $u$ must be the only function satisfying (1-5). Thus we only need to show the existence of a function $u$ which satisfies these properties. To show this we use the guaranteed existence of a minimizer and conclude that the  minimizer must satify these properties.
\begin{Lemma} \label{lemma: minimizerpiecewisebiharmonic}
Let $E \subset V_m$. The minimizer $u$ must be piecewise biharmonic on $V_m$.
\end{Lemma}
\begin{proof}
Let $\tilde{u}$ be the minimizer subject to the constraints $\tilde{u}(z) = u(z)$ and $\partial_n \tilde{u}(z) = \partial_n u(z)$ for all $z\in V_m$. Since these constraints include the original ones, we must have $T_{\zeta}(u)\leq T_{\zeta}(\tilde{u})$. Furthermore since $u$ satisfies the constraints for the new problem, we have $T_{\zeta}(\tilde{u}) \leq T_{\zeta}(u)$. By uniqueness of the minimizer we immediately get $u = \tilde{u}$. \\

By Corollary \ref{corollary: biharmonicminimizer}, we have $u = \tilde{u}$ is piecewise biharmonic.
\end{proof}

\begin{Lemma}
$\Delta u$ is continuous and has matching normal derivaties at all $y \in V_m \setminus E$, where $u$ is the minimizer.
\end{Lemma}
\begin{proof}
Since $u$ is piecewise biharmonic, we consider the piecewise biharmonic functions $v_z,w_z \in \dom_{L^2} \Delta$ for $z \in V_m \setminus E$ which satisfy
$$v_z|_{V_m} \equiv 0,$$
$$\partial_n v_z|_{V_m \setminus \{z\}} \equiv 0,$$
$$\partial_n v_z(z) = 1,$$
and
$$w_z|_{V_m \setminus \{z\}} \equiv 0,$$
$$w_z(z) = 1,$$
$$\partial_n w_z|_{V_m} \equiv 0.$$
In equation (\ref{equation:biharmonic}) we put $v=v_z$ and by Lemma \ref{lemma:eulerlagrangelaplacian} we get the continuity at $z$. For the second part, this time plugging in $v=w_z$ in equation (\ref{equation:biharmonic}) we obtain the matching normal derivatives at $z$. Hence, the result.
\end{proof}

Hence, we have fully characterized the unique minimizer as the function which is piecewise biharmonic on $V_m$ and has a Laplacian which is harmonic away from $E$.

\subsection{Constraining Values Only}

In this section, we only fix the values on $E.$ When $E \subset V_*$ we can view this as minimizing over the normal derivatives using our previous model. We have already seen existence and uniqueness. Also, the previous section guarantees that the minimizer is piecewise biharmonic. Here we find some properties of the minimizer. Throughout this section we assume $\zeta = \mu$. We write $\Delta = \Delta_{\mu}$ and $T = T_{\mu}$.

First we get the equivalent Euler Lagrange equation as before; in this case we allow more sample functions $v$ since we have relaxed the constraints. The proof is the same as for Lemma \ref{lemma:eulerlagrangelaplacian}. 
\begin{Lemma}
If $u$ is a minimizer of $T$ with respect to the constraints $u(x_i) = a_i$ then $u$ satisfies 
\begin{equation}\label{equation:eulerlagrange2}\int \Delta u \Delta v d \mu = 0\end{equation} for all $v \in \dom_{L^2}\Delta$ with $v|_E \equiv 0$.
\end{Lemma}

Using this we get two new properties of the minimizer $u$.
\begin{Lemma}The minimizer satisfies $u \in \dom \Delta$ and $\Delta u|_{V_0} = 0.$
\end{Lemma}

\begin{proof} 

Define the piecewise biharmonic function $v_z$ with 
$$v_z|_{V_m} \equiv 0,$$
$$\partial_n v_z|_{V_m \setminus \{z\}} \equiv 0,$$
$$\partial_n v_z(F_w q_i) = 1.$$

For $z \in V_m \setminus V_0$ we have $z = F_w q_i = F_{w'} q_j$.
Then applying Lemma \ref{lemma: laplacianconstruction} and 
(\ref{equation:eulerlagrange2}) we get
$\Delta (u \circ F_w)(q_i) = \Delta (u \circ F_{w'})(q_j).$ This means that $\Delta u$ is continuous so $u \in \dom \Delta$. \\

For $z=q_i \in V_0$ we apply Lemma \ref{lemma: laplacianconstruction} and 
(\ref{equation:eulerlagrange2}) again to get
$\Delta u (q_i) = 0$.
\end{proof}
\subsubsection{Calculating the Minimizing Form}
\label{section:calculations}
Using the continuity and the Dirichlet boundary conditions of $\Delta u$, we can relate $\Delta_m$ and $\Delta,$ giving us a system of equations to solve in order to calculate $T(u)$. Throughout this section, we continue to assume that $\zeta=\mu$.\\

We calculate here some basic facts to develop a system of equations to find $T(u)$. Recall that $h_i$ is the standard basis for harmonic functions on SG . From \cite{SU} we have that

$$ ~~~~~~~~~~~~~~~~~~~~~~~~~~~~M_{i,j} = \int h_i h_j d \mu~~\text{is given explicitly as}$$ $$M_{i,i} = \frac{7}{45},$$ $$~~~~~~~~~~~~~M_{i,j} = \frac{4}{45} ~~\text{for}~i\neq j.$$

Now we relate the boundary data of a biharmonic function to its Laplacian. Suppose the biharmonic function satisfies $u(q_i) = a_i$ , $\partial_n u(q_i) = b_i$ and $\Delta u(q_i)=c_i. $ 
Given such a $u$ we have
\begin{equation}\label{matrixeq}
\left(
 \begin{matrix}
  2 & -1 & -1  \\
  -1 & 2 & -1  \\
  -1 & -1 & 2  \\
 \end{matrix}\right)\left(
 \begin{matrix}
  a_0 \\
  a_1 \\
  a_2 \\
 \end{matrix}\right)
 +
 \left( \begin{matrix}
   -1 & 0 & 0 \\
   0 & -1 & 0 \\
   0 & 0 & -1 \\
 \end{matrix}\right)
 \left(
 \begin{matrix}
  b_0 \\
  b_1 \\
  b_2 \\
 \end{matrix} \right)
 = -\frac{1}{45}\left(
 \begin{matrix}
  7 & 4 & 4 \\
  4 & 7 & 4 \\
  4 & 4 & 7 \\
 \end{matrix}\right)\left(
 \begin{matrix}
  c_0 \\
  c_1 \\
  c_2 \\
 \end{matrix}\right),
 \end{equation}
 
 from which it follows that

\begin{equation}\label{equ:partialuqi}
\partial_n u(q_i) = 2u(q_i) - u(q_{i+1}) - u(q_{i-1}) + \frac{1}{45}(7\Delta u(q_i) + 4 \Delta u(q_{i+1}) + 4 \Delta u(q_{i-1}))
\end{equation} 
and
 $$\Delta u(q_i) = -15(2u(q_i) - u(q_{i+1}) - u(q_{i-1}))+11\partial_nu(q_i) - 4\partial_nu(q_{i+1}) -4\partial_nu(q_{i-1}).$$
 
 \begin{Lemma}
Given any $z=F_w(q_i)=F_{w'}(q_{j}) \in V_m \setminus V_0$, we have 
  \[
 \Delta_m u(z) = \frac{1}{45\cdot5^m}(14 \Delta u(z) + 4 \sum_{y \underset{m}{\sim} z} \Delta u (y)).
 \]
 \end{Lemma}

 \begin{proof}
Replacing $u$ by $u\circ F_w$ in  equation (\ref{equ:partialuqi}), we have

 $$-\partial_n (u \circ F_w)(q_i)  =  - 2 (u\circ F_w)(q_i) + (u\circ F_w)(q_{i-1}) + (u\circ F_w)(q_{i+1})$$
$$ - \frac{1}{45} (7 \Delta (u\circ F_w)(q_i)+ 4 \Delta (u\circ F_w)(q_{i-1}) + 4 \Delta (u\circ F_w) (q_{i+1})) $$
and

$$-\partial_n (u \circ F_{w'})(q_j)  =  - 2 (u\circ F_{w'})(q_j) + (u\circ F_{w'})(q_{j-1}) + (u\circ F_{w'})(q_{j+1})$$
$$ - \frac{1}{45} (7 \Delta (u\circ F_{w'})(q_j)+ 4 \Delta (u\circ F_{w'})(q_{j-1}) + 4 \Delta (u\circ F_{w'}) (q_{j+1})). $$
 
 Then, summing them up and using $\Delta(u \circ F_w)=r_{w} \mu_{w}\Delta u \circ F_w$, $r_w=\left(\frac{3}{5} \right)^m~ \text{and}~ \mu_w=\left(\frac{1}{3}\right)^m$ 
 yields the result.
 \end{proof}
 
 Then we have $|V_m\setminus V_0|$ number of equations and we know that $\Delta u|_{V_0}=0$ so we can solve the above system of equations and express $\Delta u |_{V_m\setminus V_0}$ in terms of $u|_{V_m}$.
By Gauss-Green, the fact that $\Delta u$ is piecewise harmonic, and the fact $\Delta u$ vanishes on the boundary we have 
\begin{align*}
\int_{SG} |\Delta u(x)|^2 d \mu(x)=-\E(\Delta u,u)&=-\sum_{|w|=m}r_w^{-1}\E(\Delta u \circ F_w,u \circ F_w)\\
&=-r^{-m}\sum_{|w|=m}{\E_0(\Delta u \circ F_w, u \circ F_w)}\\
&=-r^{-m}\sum_{V_m \setminus V_0}{\Delta u(x) \left(\sum_{y \underset{m}{\sim} x}{u(y) - u(x)}\right)}\\
&=r^{-m}\sum_{V_m \setminus V_0}\Delta u \Delta_m u.
\end{align*}

\subsubsection{Generalization to all of $SG$}
In this section, we present a construction for the minimizer when $E = \{x_1, \dots, x_n\} \subset SG$ is not restricted to the junction points. As before we require Dirichlet boundary conditions. This construction parallels the construction for the energy minimizer. Note that Theorem \ref{theorem:laplacian1} generalizes immediately to all $E \subset SG$ so we get uniqueness of the minimizer.

\begin{Lemma} \label{generallaplacianeuler}
Suppose $u(x)=\sum_{i=1}^{n}  c_i \int G(x,y)G(x_i,y) d\mu (y)$ for some constants  $c_i$. Then for all $v \in dom_{L^2}\Delta$ such that $v|_{V_0 \cup E}=0$,we have $$\int \Delta u \Delta v d \mu=0. $$

\end{Lemma}

\begin{proof}
For such $v$, we have $$\int \Delta u \Delta v d \mu=\int (-\sum c_i G(x_i,x))(\Delta v(x)) d\mu(x)=\sum_{i} c_i v(x_i)=0.$$
\end{proof}

Thus if we can find a function of the above form satisfying the constraints, then we get a function which satisfies the neccessary condition so it is the unique minimizer.

\begin{Lemma}
$$[G]_{i,j}=\int G(x_i,y)G(x_j,y) d \mu (y)$$ is invertible. Thus we can find $c_i$ such that $u=\sum_{i=1}^{n}  c_i \int G(x,y)G(x_i,y) d \mu (y)$ satisfies the constraints.
\end{Lemma}

\begin{proof}
Suppose we have $c_i$ with $u=\sum_{i=1}^{n}  c_i \int G(x,y)G(x_i,y) d\mu (y)$ such that $u_{E \cup V_0}=0$. Then we apply Lemma \ref{generallaplacianeuler} with $v=u$ to get $\int |\Delta u|^2 d \mu=0.$ This implies $u$ is a harmonic function but it vanishes at boundary so it is zero. This proves the injectivity of $G$, which gives invertibility (and surjectivity). We can solve for $c_i$ such that $u(x_i) = a_i$.
\end{proof}

\vspace{20pt}

\begin{small}
\begin{flushleft}
\textbf{Pak Hin Li}\\
\textsc{Department of Mathematics, Lady Shaw Building, The Chinese University of Hong Kong, Shatin, Hong Kong.} \\
\textsc{E-mail:}\ \textbf{pakhinbenli@gmail.com} 

\vspace{5pt}

\textbf{Nicholas Ryder}\\
\textsc{Department of Mathematics, Rice University, 6100 Main St, Houston, TX 77005 USA.}\\
\textsc{E-mail:}\ \textbf{nick.ryder@rice.edu}

\vspace{5pt}

\textbf{Robert S. Strichartz}\\
\textsc{Department of Mathematics, Cornell University, Ithaca, NY 14853, USA.}\\
\textsc{E-mail:}\ \textbf{str@math.cornell.edu}

\vspace{5pt}

\textbf{Baris Evren Ugurcan}\\
\textsc{Department of Mathematics, Cornell University, Ithaca, NY 14853, USA.}\\
\textsc{E-mail:}\ \textbf{beu4@cornell.edu} 
\end{flushleft}
\end{small}

\end{document}